\documentclass[journal]{IEEEtran}
\usepackage[cmex10]{amsmath}
\usepackage{amsfonts}
\usepackage{amssymb}
\usepackage{mathrsfs}
\usepackage{graphicx}
\usepackage{cite}
\usepackage{float}
\usepackage{setspace}
\newtheorem{theorem}{Theorem}

\IEEEoverridecommandlockouts
\newcommand{\Fig}{Fig.}
\newcommand{\Figs}{Figs.}
\begin{document}
%
%
\title{Directed Self-Assembly of Linear Nanostructures by Optimal Control of External Electrical Fields}
%
%
\author{Arash Komaee and Paul I. Barton
\thanks{This work was supported by the National Science Foundation under Grant No. CBET-1033533.}
\thanks{A. Komaee was with the Department of Chemical Engineering, Massachusetts Institute of Technology, Cambridge, MA 02139, USA. He is now with the Department of Electrical and Computer Engineering, Southern Illinois University, Carbondale, IL 62901, USA (e-mail: akomaee@siu.edu).}
\thanks{P.~I.~Barton is with the Department of Chemical Engineering, Massachusetts Institute of Technology, Cambridge, MA 02139, USA (e-mail: pib@mit.edu).}
}
\maketitle
%
%
\begin{abstract}
An optimal control strategy is developed to construct nanostructures of desired geometry along line segments by means of directed self-assembly of charged particles. Such a control strategy determines the electric potentials of a set of electrodes located at fixed points in the line segment. The particles move under the electric forces generated by these electrodes and by the interactions between the particles themselves to form a desired pattern eventually. Due to technology limitations, the particle positions cannot be measured during the course of control, so that the control is open-loop in nature. Such an open-loop control optimally changes the electrode potentials in time in order to create a desired pattern with the highest probability, despite the inherent uncertainty in the initial positions and the dynamical behaviors of the particles. Two models are proposed to describe the uncertain dynamics of the particles: a continuous model relying on a set of nonlinear stochastic differential equations, and a discrete Ising model consisting of a large dimensional continuous-time Markov chain. While the first model is more mathematically tractable, the second one more precisely describes particles at the nanometer scale. The control design procedure begins with the continuous model and identifies the structure of its stable equilibria, which is used later to propose a piecewise constant structure for the control and to demonstrate that the optimal value of each piece is independently obtained from a certain static optimization problem. It is shown next that the design procedure can be applied to the discrete model with only minor modifications. A numerical example of control design is presented.
\end{abstract}
%
%
%
%
\section{Introduction}
Self-assembly is the process of forming an ordered structure from initially disordered components that only interact locally, without external direction. At the molecular level, this process is a common technique for fabrication of nanostructures with periodic patterns~\cite{ART.WhitesidesGrzybowski.02,ART.WhitesidesBoncheva.02,ART.AmirParvizEtAl.03,ART.Zhang.03,ART.KimEtAl.03,ART.ParkEtAl.03,
ART.ChengEtAl.04,ART.LoveEtAl.05,ART.KhaledEtAl.05,ART.OzinEtAl.09}. Due to the important role of this fabrication technique in nanotechnology, several researchers have studied self-assembly phenomena at a theoretical level based on abstract models~\cite{TECH.Adleman.00,PROC.AdlemanEtAl.01,COLL.CarboneSeeman.04,ART.SoloveichikWinfree.07,
ART.MajumderEtAl.08,ART.HormozBrenner.11,ART.Chandran.12,ART.Patitz.13}.

Self-assembled nanostructures usually demonstrate periodic patterns that only depend on the nature of their components and the environmental conditions under which the patterns are formed. However, several applications require fabrication of nanostructures with certain non-periodic geometries~\cite{ART.RosiMirkin.05,ART.StephanopoulosEtAl.05,ART.WinklemanEtAl.05,ART.StoykovichEtAl.07,PROC.Kiehl.07,ART.LalanderEtAl.10}. Given the major role of molecular self-assembly in fabrication of periodic nanostructures, it is reasonable to ask if this process can be externally directed to fabricate nanostructures of desired geometry which are not necessarily periodic. Such a \textit{directed self-assembly} process is the focus of this paper.

In directed self-assembly, a number of charged nanoparticles (e.g., DNA tiles) are manipulated by external electrical fields to form a nanostructure of desired geometry. The directing electrical fields are generated and controlled by relatively small number of electrodes (compared to the number of particles) located at fixed locations on the substrate containing the particles. The dynamics of the particles are primarily governed by the interactions between them (self-assembly), and is modified to some extent by manipulation of the electrical potentials of these electrodes (external direction). The particles are initially distributed randomly on the substrate and are perturbed by random disturbances during the assembly process. Since the particle positions cannot be measured during the course of control, a feedback loop cannot be established and the electrodes are actuated only by open-loop controls.

The control objective is to direct the particles towards formation of a desired pattern despite the uncertainty in their dynamics and initial positions. Under an optimal design, this control must maximize the probability of forming the desired pattern by the end of the assembly process, and maintain the formed structure under a static control afterward. Such a constant static control creates a stable equilibrium representing the desired pattern. In addition to this intended stable equilibrium, the static control inherently creates multiple undesired stable equilibria, and a major challenge of an optimal control is to prevent the system falling into such \textit{kinetic traps}. Given the large number of these kinetic traps and the inherent uncertainty in the initial distribution and dynamics of the particles, the system will most likely be trapped by an undesired stable equilibrium (formation of a wrong pattern), unless a phase of dynamic control (time-dependent) is applied prior to the static control.

This paper intends to develop an analytic framework for study of directed self-assembly, including a systematic method for control design. Rather than focusing on a detailed model of the physical process, the main emphasis is on providing a clear  understanding of the fundamental concepts such as static control, kinetic traps, and dynamic control. To be consistent with this approach, the models and control problem considered in this paper are abstractions of real-world directed self-assembly: they capture the essence of this phenomenon but do not reflect all its details. In particular, the paper focuses on directed self-assembly of linear structures (one-dimensional patterns along straight lines), a simplified model also adopted in \cite{PROC.AdlemanEtAl.01,ART.Chandran.12} to study ``undirected'' self-assembly. This special case of the more general planar patterns demonstrates certain properties that facilitate exact characterization of the stable equilibria of the system, which in turn, allows for a rigorous analysis and control design methodology. The concepts and methods developed here for linear patterns are equally valid in two dimensions, while generalization of some computational procedures might not be immediate. Such a generalization, at least approximately, is the subject of our future work.

This paper adopts two different but closely related models to describe directed self-assembly. A continuous model is presented in Section~\ref{LinearGeometry.ModelSection} which allows the particles to position continuously at any arbitrary point in a line segment. This model is precise for larger particles of micrometer diameter and its continuous nature facilitates our analysis in Section~\ref{LinearGeometry.ControlDesignSection}. Later in Section~\ref{LinearGeometry.DiscreteModelSection}, a discrete model is presented for nano-scale particles, and it is shown how the results of Section~\ref{LinearGeometry.ControlDesignSection}, originally developed for the continuous model, can be tailored to this discrete model only with minor modifications. Our main results on the structure of the kinetic traps, control design, and optimization of the electrodes are presented in Section~\ref{LinearGeometry.ControlDesignSection}.
%
%
\section{Model and Problem Statement}\label{LinearGeometry.ModelSection}
The system of particles considered in this paper is described at the nano-scale ($\sim10\mbox{nm}$) by a discrete Ising model and a master equation~\cite{ART.SolisEtAl.10A,ART.SolisEtAl.10B,ART.LakerveldEtAl.12}. In this model, the particles can occupy only a finite set of positions along a line segment, in contrast to a continuous model used for larger particles ($\sim1\mu\mbox{m}$) in which the particles can position continuously at any arbitrary point along the line segment. The latter model is directly derived from the classical Newton's second law of motion, and Coulomb's law that governs the interactions between the particles and the forces applied to the particles by the electrodes. Such a continuous model is more intuitive and mathematically more tractable, thus it is a convenient point of departure to explain the concepts and control design methodology developed in this paper. We begin with this continuous model and construct our control design method on this basis. Later in Section~\ref{LinearGeometry.DiscreteModelSection}, we present the discrete model and show that for the purpose of control design using our proposed method, the two models are mathematically equivalent and can be interchanged with minor modifications. In particular, our control design relies on the steady-state behaviors of these models which match closely despite their different dynamical behaviors. It is emphasized that the two models describe different physical phenomena and they are not necessarily interchangeable for other purposes such as simulations.

Throughout this paper, the time-dependent state and control vectors are shown by the boldface letters $\mathbf{x}$ and $\mathbf{u}$, so that $\mathbf{x}$ and $\mathbf{u}$ are mappings from time into the state space and the control set, respectively. The values of the state and the control vectors at time $t$ are denoted by $\mathbf{x}\left(t\right)$ and $\mathbf{u}\left(t\right)$ or simply by~$x$ and $u$ as a shorthand. All constant vectors and other functions of time or other variables are shown in plain letters.

Referring to \Fig~\ref{LinearGeometry.ControlsParticles}, consider a line segment and assume that $c+1$ electrodes are located at the fixed points
\begin{equation*}
0=q_0<q_1<\cdots<q_c
\end{equation*}
in this line segment. Suppose that ${n}$ identical charged particles are located between $q_0$ and $q_c$ at the points $x_1,x_2,\ldots,x_{n}$. The particle positions specify the state vector $x=\left(x_1,x_2,\dots,x_{n}\right)$ in $\mathbb{R}^n$. The control vector $u=\left(u_0,u_1,u_2,\dots,u_c\right)$ is defined in such a manner that its $k$th component $u_k$ represents the electric charge of the electrode $k=0,1,,\ldots,c$ normalized to the charge of a single particle. It is assumed that at any time, the value of the control vector $u$ can be arbitrarily chosen within the control set $\mathcal{U}\subset\mathbb{R}^{c+1}$.
\begin{figure}[h]
\begin{center}
\includegraphics[scale=1]{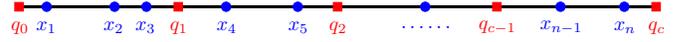}\\
\caption{Geometry of the charged particles and the control electrodes along a line segment. The disks represent the particles while the boxes stand for the electrodes.}\label{LinearGeometry.ControlsParticles}
\end{center}
\end{figure}

The total energy associated with the state $x$ of the particles and the control value $u$ is given by ${\kappa}V\left(x,u\right)$, where the normalized energy function $V:\mathbb{R}^n\times\mathbb{R}^{c+1}\to\mathbb{R}$ is given~by
\begin{equation}\label{LinearGeometry.EnergyFunction}
V\left(x,u\right)=\frac{1}{2}\sum_{i=1}^{{n}}\sum_{\substack{j=1\\j\neq{i}}}^{{n}}
\frac{1}{\left|x_i-x_j\right|}+\sum_{i=1}^{{n}}\sum_{j=0}^{c}\frac{u_j}{\left|x_i-q_j\right|}\,.
\end{equation}
Here, $\kappa>0$ is a constant defined as
\begin{equation*}
\kappa=\frac{e^2}{4\pi\varepsilon_0\varepsilon}\,,
\end{equation*}
where $e$ denotes the charge of a single particle, $\varepsilon_0$ stands for the permittivity of free space, and the dimensionless constant~$\varepsilon$ is the relative permittivity of the environment containing the particles. The negative gradient $-{\kappa}\nabla_xV\left(x,u\right)$ of the total energy ($\nabla_x$ denotes the gradient operator with respect to the first argument) is a vector in $\mathbb{R}^n$ whose $k$th component is the total force applied to the $k$th particle by the remaining $n-1$ particles (first term on the right-hand side of~\eqref{LinearGeometry.EnergyFunction}) and by the $c+1$ electrodes (second term on the right-hand side of~\eqref{LinearGeometry.EnergyFunction}).

Assume that the particles start from the initial state $x_0$ at time $t=0$ and their state $\mathbf{x}\left(t\right)\in\mathbb{R}^n$ evolves in time under a time-varying control $\mathbf{u}\left(t\right)\in\mathcal{U}$. The dynamics of the particles is determined by three factors: the Coulomb forces caused by the interactions between the particles and the electrodes and the interactions between the particles themselves, the friction between the particles and their surrounding fluid (drag), and the Brownian motion. In the absence of the Brownian motion, the particles accelerate under the Coulomb forces and the opposing resistance of the surrounding fluid. By Stokes' drag law~\cite{BOOK.Probstein.94}, such resistive forces are negatively proportional to the velocity of the particles with a proportionality constant $\mu>0$. In response to a sudden change in the control Coulomb forces, the particles accelerate for a short period of time before the opposing drag forces balance this change in the control forces. In a large friction regime, this acceleration period is short and negligible~\cite{ART.Felderhof.87}, so that it is a reasonable approximation to take the drag and the Coulomb forces as equal. Then the velocity of each particle will be proportional to its applied Coulomb force (Smoluchowski approximation). By normalizing time to $\mu/\kappa$, the proportionality constant is unit and the equation of motion of the particles can be simply written as
\begin{equation}\label{LinearGeometry.DeterministicEquationMotion}
\dot{\mathbf{x}}\left(t\right)=-\nabla_xV\left(\mathbf{x}\left(t\right),\mathbf{u}\left(t\right)\right).
\end{equation}

The contribution of the Brownian motion is incorporated into the equation of motion using a ${n}$-dimensional standard Wiener process $\left\{\mathbf{w}\left(t\right)\right\}$ as described by the It\^{o} stochastic differential equation~\cite{BOOK.Oksendal.03}
\begin{equation}\label{LinearGeometry.StochasticEquationMotion}
d\mathbf{x}\left(t\right)=-\nabla_xV\left(\mathbf{x}\left(t\right),\mathbf{u}\left(t\right)\right)dt+\sigma{d\mathbf{w}\left(t\right)}.
\end{equation}
Here, $\sigma=\sqrt{2{\kappa}k_BT}$ is a constant depending on the Boltzmann constant $k_B$, the temperature $T$ in Kelvin, and the normalizing factor $\kappa$ of the energy function. It is assumed that the initial state $\mathbf{x}\left(0\right)=x_0$ is a random vector with the known probability density function $p_0\left(x\right)$ satisfying
\begin{equation*}
p_0\left(x\right)=0,{\quad}x\notin\left[q_0,q_c\right]^n.
\end{equation*}
The stochastic differential equation~\eqref{LinearGeometry.StochasticEquationMotion} represents the Langevin equation for the particle positions~\cite{ART.UhlenbeckOrnstein.30,ART.Roux.92}.

Suppose that the interval $\left[q_0,q_c\right]$ is partitioned into $N$ subintervals $\mathscr{I}_1,\mathscr{I}_2,\ldots,\mathscr{I}_N$ of the equal length $d_0$. It is assumed that the number $N$ of these subintervals is larger than the number ${n}$ of the particles. Further, assume that the distance $q_k-q_{k-1}$ between the electrodes is an integer multiple of the grid size $d_0$.

A pattern $\mathcal{P}\in\left\{0,1\right\}^N$ is defined as a binary vector of dimension $N$ with exactly ${n}$ ones ($1$'s) and $N-{n}$ zeros ($0$'s). The total number of patterns is given by the combination $S=\left(\substack{N\\{n}}\right)$. Each binary component of a pattern represents one of the subintervals $\mathscr{I}_k$. It is said that a pattern $\mathcal{P}$ is formed by the particles, if exactly one particle is inside the subintervals associated with the components of value $1$ in $\mathcal{P}$. Notice that every state of the particles does not necessarily define a pattern since it is possible that more than one particle belong to a certain subinterval. \Fig~\ref{LinearGeometry.FigDynamicControlExample} illustrates a nanostructure created by $n=8$ particles in a grid of $N=16$ cells with $c+1=5$ electrodes. The binary vector $\mathcal{P}=\left(0,1,1,1,0,0,1,1,0,0,1,0,0,1,0,1\right)$ represents this nanostructure (pattern) in such a manner that each occupied cell corresponds to a $1$ in this vector.
\begin{figure}[h]
\begin{center}
\includegraphics[scale=1]{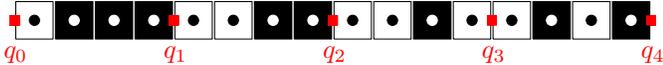}\\
\caption{Nano-structure created by $n=8$ particles in a grid of $N=16$ cells with $c+1=5$ electrodes. The pattern is represented by the binary vector $\mathcal{P}=\left(0,1,1,1,0,0,1,1,0,0,1,0,0,1,0,1\right)$, where each filled cell is corresponding to a $1$ in this vector. The boxes mark the locations of the electrodes and the disks show the locations that the particles can occupy in the discrete model of Section~\ref{LinearGeometry.DiscreteModelSection}.}\label{LinearGeometry.FigDynamicControlExample}
\end{center}
\end{figure}

Each pattern $\mathcal{P}$ is uniquely mapped into a subset $\mathscr{P}_0\left(\mathcal{P}\right)$ of the state space $\left[q_0,q_c\right]^{n}$ such that the formation of that pattern at time $t$ occurs if $\mathbf{x}\left(t\right)\in\mathscr{P}_0\left(\mathcal{P}\right)$. Let $\iota_k$, $k=1,2,\ldots,n$ denote the indices of the $k$th $1$ in the binary vector (pattern)~$\mathcal{P}$. Then the value of the mapping $\mathscr{P}_0\left(\mathcal{P}\right)$ is defined as the union
\begin{equation*}
\mathscr{P}_0\left(\mathcal{P}\right)=\bigcup\mathscr{I}_{i_1}\times\mathscr{I}_{i_2}\times\cdots\times\mathscr{I}_{i_{n}}
\end{equation*}
taken over the set of all ${n}!$ permutations of $\left(\iota_1,\iota_2,\ldots,\iota_{n}\right)$.

The control goal is to move the particles in such a manner that they form a desired pattern $\mathcal{P}_d$ with the highest probability at a final time $t_f$. This must be achieved despite the inherent uncertainty in the system dynamics and the initial state and by means of an open-loop control since the particle positions (components of the state vector) cannot be measured during the course of control due to technology limitations. Therefore, the objective is to obtain an open-loop control $\mathbf{u}\left(t\right)\in\mathcal{U}$ on $t\in\left[0,t_f\right]$ to form a desired pattern $\mathcal{P}_d$ at the final time $t_f$ with the highest possible probability, and to maintain this maximum probability under the constant control $u_\mathrm{ss}\triangleq\mathbf{u}\left(t_f\right)$ afterward ($t>t_f$). These requirements are mathematically expressed by maximizing the payoff function
\begin{equation}\label{LinearGeometry.MaximumProbabilityCondition}
J=\Pr\left\{\mathbf{x}\left(t_f\right)\in\mathscr{P}_0\left(\mathcal{P}_d\right)\right\}
\end{equation}
under the inequality constraint
\begin{equation}\label{LinearGeometry.SteadyStateProbabilityCondition}
\left|\Pr\left\{\mathbf{x}\left(t\right)\in\mathscr{P}_0\left(\mathcal{P}_d\right)\right\}-\lim_{t^\prime\to+\infty}
\Pr\left\{\mathbf{x}\left(t^\prime\right)\in\mathscr{P}_0\left(\mathcal{P}_d\right)\right\}\right|\leqslant\epsilon
\end{equation}
for all $t\geqslant{t_f}$ and some small $0<\epsilon<1$. The final time $t_f>0$ is a free parameter that is preferred but not constrained to be reasonably short.

This optimization problem can be formulated as an optimal control problem with deterministic but infinite-dimensional dynamics. It is well known that the probability density function $p\left(x,t\right)$ of $\mathbf{x}\left(t\right)$ solves the Fokker-Planck equation~\cite{BOOK.Soize.94}
\begin{equation*}
\frac{{\partial}p}{\partial{t}}\left(x,t\right)=\nabla_x\cdot\left(\nabla_xV\left(x,\mathbf{u}\left(t\right)\right)p\left(x,t\right)
+\frac{1}{2}\;\!\sigma^2\nabla_xp\left(x,t\right)\right)
\end{equation*}
with the initial condition $p\left(x,0\right)=p_0\left(x\right)$, where $\nabla_x\cdot$ denotes the divergence operator with respect to the first argument. Then, subject to this infinite-dimensional dynamics, an admissible control $\mathbf{u}\left(t\right)\in\mathcal{U}$ is sought on $t\in\left[0,t_f\right]$ to maximize the payoff function
\begin{equation*}
J=\int_{\mathscr{P}_0\left(\mathcal{P}_d\right)}p\left(x,t_f\right)dx
\end{equation*}
while maintaining the terminal condition
\begin{equation*}
\left|\int_{\mathscr{P}_0\left(\mathcal{P}_d\right)}p\left(x,t_f\right)dx-\lim_{t\to+\infty}
\int_{\mathscr{P}_0\left(\mathcal{P}_d\right)}p\left(x,t\right)dx\right|\leqslant\epsilon
\end{equation*}
under the constant control $\mathbf{u}\left(t_f\right)$ for $t>t_f$.

This new formulation represents a standard optimal control problem\footnote{The payoff function consists of only a terminal payoff and does not include an integral over time of the state.} although its solution is complicated by the infinite dimensions of the Fokker-Planck equation. In Section~\ref{LinearGeometry.ControlDesignSection}, we obtain a solution to this optimal control problem within a certain class of piecewise constant controls. This solution exploits a certain structure of the nonlinear system~\eqref{LinearGeometry.StochasticEquationMotion} to convert the optimal control problem above to a sequence of static optimization problems with tractable computational complexity.
%
%
\section{Control Design}\label{LinearGeometry.ControlDesignSection}
Our control design procedure consists of two steps: design of a static control, and design of a dynamic control. The static control $u_\mathrm{ss}\in\mathcal{U}$ is a constant control intended to create a stable equilibrium $x_\mathrm{ss}\in\mathscr{P}_0\left(\mathcal{P}_d\right)$ inside the subset $\mathscr{P}_0\left(\mathcal{P}_d\right)$ of the state space that represents a desired pattern $\mathcal{P}_d$. In the absence of the Brownian motion, a stable equilibrium is a point of the state space with a sustainable balance of forces under which the particles are at rest, i.e., the state vector $\mathbf{x}\left(t\right)$ settles at this point in the steady-state so that $\dot{\mathbf{x}}\left(t\right)=\mathbf{0}$. In the presence of the Brownian motion, the state vector moves towards the stable equilibrium and eventually reaches a stationary regime under which it randomly jitters in the vicinity of $x_\mathrm{ss}$. The desired pattern is formed in this regime as $x_\mathrm{ss}$ is inside $\mathscr{P}_0\left(\mathcal{P}_d\right)$ and the state vector remains close to this point. The optimal design of the static control $u_\mathrm{ss}$ is discussed in Section~\ref{LinearGeometry.StaticControlSection}.

The static control $u_\mathrm{ss}$ and the corresponding equilibrium $x_\mathrm{ss}$ must jointly satisfy the conditions\footnote{The notation $H\succ0$ indicates that $H$ is a positive definite matrix.}
\begin{subequations}
\begin{align}
-&\nabla_xV\left(x_\mathrm{ss},u_\mathrm{ss}\right)=\mathbf{0}\label{LinearGeometry.EquilibriumEquation}\\
&H\left(x_\mathrm{ss},u_\mathrm{ss}\right)\succ0\label{LinearGeometry.PositiveDefiniteHessian}\\
&x_\mathrm{ss}\in\mathscr{P}_0\left(\mathcal{P}_d\right)\\
&u_\mathrm{ss}\in\mathcal{U},
\end{align}
\end{subequations}
where $H\left(x_\mathrm{ss},u_\mathrm{ss}\right)\in\mathbb{R}^{n\times{n}}$ is the Hessian matrix
\begin{equation*}
H\left(x_\mathrm{ss},u_\mathrm{ss}\right)=\frac{\partial^2V}{\partial{x^\mathrm{T}}\partial{x}}\left(x_\mathrm{ss},u_\mathrm{ss}\right).
\end{equation*}
For a fixed $u_\mathrm{ss}$, the solution $x_\mathrm{ss}$ of the algebraic equation~\eqref{LinearGeometry.EquilibriumEquation} is a stationary point of the energy function $V\left(\,\cdot\,,u_\mathrm{ss}\right)$, and if this stationary point is a strict local minimum of $V\left(\,\cdot\,,u_\mathrm{ss}\right)$, it is a stable equilibrium of the deterministic dynamical system~\eqref{LinearGeometry.DeterministicEquationMotion}. As noted in~\eqref{LinearGeometry.PositiveDefiniteHessian}, if the Hessian matrix $H\left(x_\mathrm{ss},u_\mathrm{ss}\right)$ is positive definite, the stationary point is a strict local minimum.

For any given static control $u_\mathrm{ss}$, the algebraic equation~\eqref{LinearGeometry.EquilibriumEquation} can have multiple stable solutions for $x_\mathrm{ss}$, not necessarily inside $\mathscr{P}_0\left(\mathcal{P}_d\right)$ to form a desired pattern. This is caused by the fact that the energy function $V\left(\,\cdot\,,u_\mathrm{ss}\right)$ can have multiple strict local minima (see \Fig~\ref{LinearGeometry.MultipleEnergyMinima}(a)) that allow for the formation of multiple stable patterns. As shown in \Fig~\ref{LinearGeometry.MultipleEnergyMinima}(a), the energy function consists of several \textit{potential wells} with a single stable equilibrium (a strict local minimum) at the bottom of each one. Each potential well specifies a \textit{region of attraction} (ROA)---an open subset of the state space containing exactly one stable equilibrium and marked by the property that if the state $\mathbf{x}\left(t\right)$ of the dynamical system~\eqref{LinearGeometry.DeterministicEquationMotion} is initially inside a certain ROA, it remains inside that ROA and moves towards its equilibrium. This property is demonstrated by the inequality
\begin{align*}
\frac{d}{dt}V\left(\mathbf{x}\left(t\right),u_\mathrm{ss}\right)&
=\dot{\mathbf{x}}^\mathrm{T}\left(t\right)\nabla_xV\left(\mathbf{x}\left(t\right),u_\mathrm{ss}\right)\nonumber\\
&=-\left\|\nabla_xV\left(\mathbf{x}\left(t\right),u_\mathrm{ss}\right)\right\|^2\\
&\leqslant0,
\end{align*}
where the equality holds if and only if $\mathbf{x}\left(t\right)$ is an equilibrium. This implies that the total energy inside a single ROA monotonically decreases before the system settles at the stable equilibrium at the bottom of the potential well. Since the energy level inside a ROA can never exceed its initial value, the state vector cannot escape its initial ROA.
\begin{figure}[h]
\begin{center}
\includegraphics[scale=1]{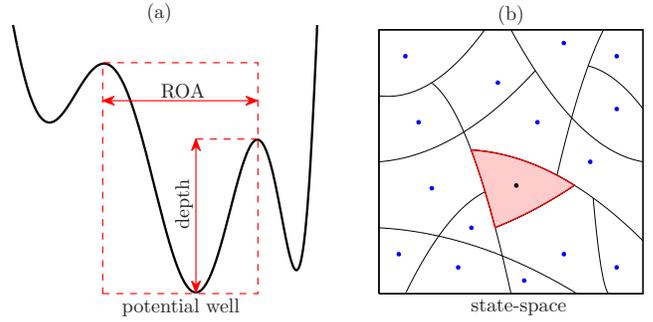}\\
\caption{Multiple equilibria, regions of attraction, and potential wells: (a) multiple
potential wells of a dynamical system of one dimension; (b) the state space is partitioned by the ROAs. The depth of a potential well is marked in (a) as the energy difference between the deepest point of that well and the lowest energy level on its boundary. In (b), each ROA contains a single stable equilibrium represented by a dot and the shaded region specifies the ROA containing the desired equilibrium. Without a dynamic control, only the initial states inside this ROA lead to a desired pattern.}\label{LinearGeometry.MultipleEnergyMinima}
\end{center}
\end{figure}

For a static control with multiple stable equilibria, the state space is partitioned\footnote{This means that the ROAs are disjoint subsets of the state space whose union is the entire state space excluding the boundaries of the ROAs.} by the set of ROAs as illustrated schematically in~\Fig~\ref{LinearGeometry.MultipleEnergyMinima}(b). However, only a certain equilibrium forms the desired pattern, and only those initial states belonging to the ROA of that equilibrium end up with the desired pattern (see~\Fig~\ref{LinearGeometry.MultipleEnergyMinima}(b)). Thus, before starting the phase of static control, it is necessary to bring the initial state inside the desired ROA. This task is performed by a dynamic control, a time-varying open-loop control that drives the state vector $\mathbf{x}\left(t\right)$ towards the desired ROA regardless of the inherent uncertainty in the initial state. Notice that the state vector---particle locations---is not known to the controller during the course of control. It is shown in Section~\ref{LinearGeometry.DynamicControlSection} that the dynamic control can be decomposed into a sequence of static controls, so that a piecewise constant structure is proposed for this control.

Before proceeding with the control design in Sections~\ref{LinearGeometry.StaticControlSection} and~\ref{LinearGeometry.DynamicControlSection}, the structure of the ROAs and their stable equilibria is studied in Section~\ref{LinearGeometry.ROAStructureSection}.
%
%
\subsection{Structure of the Regions of Attraction}\label{LinearGeometry.ROAStructureSection}
In the deterministic system \eqref{LinearGeometry.DeterministicEquationMotion} under the constant control $\mathbf{u}\left(t\right)=u_\mathrm{ss}$, the state vector remains in the same ROA that it takes its initial value. This property does not generally hold for stochastic systems perturbed by a Wiener process. For such stochastic systems, the random disturbance causes the state vector to jitter around one of the equilibria, often inside the same ROA. Occasionally, the deviations from the equilibrium are large enough to drive the state vector outside that ROA, allowing other stable equilibria to attract it. This migration from one potential well to another can be viewed as being caused by  a high level of energy absorbed from a disturbance that exceeds the depth of the departure potential well. As illustrated in \Fig~\ref{LinearGeometry.MultipleEnergyMinima}(a), the depth of a potential well is defined as the energy difference between its deepest point and the lowest energy on its boundary. In the stochastic system~\eqref{LinearGeometry.StochasticEquationMotion}, however, the state vector cannot leave its initial ROA (almost surely), even though the system is disturbed by a Wiener process. This unusual property is a consequence of the infinite depth of the potential wells in this system.

To show this property, consider an electrode charged with the same polarity as the particles (assumed positive), and a particle that is pushed towards the electrode by an external force. For example, in \Fig~\ref{LinearGeometry.ControlsParticles} suppose that the particle located at $x_1$ is pushed left towards the electrode at $q_0=0$. The external force required to maintain the particle at the distance $x_1$ from the electrode is proportional to $1/x_1^2$ which increases unboundedly as $x_1$ tends to $0$. This implies that the particle cannot reach the electrode using a bounded external force, and clearly cannot pass through it. With a similar argument, two particles cannot hit each other under bounded external forces that squeeze them together.

Assume that all electrodes have constant positive charges (same polarity as the particles), i.e., $\mathbf{u}\left(t\right)=u_\mathrm{ss}$ is a vector of positive components. Let $\nu_k$, $k=1,2,\ldots,c$ be the number of particles in the interval $\left(q_{k-1},q_k\right)$ so that $\nu_1+\nu_2+\cdots+\nu_c={n}$. Based on the above argument, the integers $\nu_k$ remain constant over time, i.e., at any time after applying the constant control $\mathbf{u}\left(t\right)=u_\mathrm{ss}$, the number of particles in the interval $\left(q_{k-1},q_k\right)$ is equal to its value just before application of this control. In addition, the order of the particles is preserved over time as the particles cannot jump over each other.

Since the particles are identical and their order does not change over the course of control, it can be assumed without loss of generality that they are labeled by $1,2,\ldots,{n}$ from left to right, as shown in \Fig~\ref{LinearGeometry.ControlsParticles}. This requires the initial distribution of the state vector to satisfy
\begin{equation*}
p_0\left(x\right)=0,{\quad}x\notin\mathscr{S}_0
\end{equation*}
where the simplified state space $\mathscr{S}_0$ is defined as
\begin{equation*}
\mathscr{S}_0=\left\{x|\,q_0<x_1<x_2<\cdots<x_{n}<q_c\right\}.
\end{equation*}
In addition, the fixed order of the particles allows the mapping $\mathscr{P}_0\left(\mathcal{P}\right)$ to be simplified to $\mathscr{P}\left(\mathcal{P}\right)$ defined as
\begin{equation*}
\mathscr{P}\left(\mathcal{P}\right)=\mathscr{I}_{\iota_1}\times\mathscr{I}_{\iota_2}\times\cdots\times\mathscr{I}_{\iota_{n}},
\end{equation*}
where $\iota_k$, $k=1,2,\ldots,n$ denote the indices of the $k$th $1$ in the pattern $\mathcal{P}$.

Let $\nu=\left(\nu_1,\nu_2,\ldots,\nu_c\right)$ be a vector in $\mathbb{N}_0^c$ whose $k$th component is the number of particles in the interval $\left(q_{k-1},q_k\right)$. Since the total number of particles is $n$, this vector must satisfy the constraint $\left\|\nu\right\|_1=n$. The total number of such vectors is the ``\textit{weak compositions of $n$ into $c$ parts}'' \cite[Thm.~5.2]{BOOK.Bona.06} and is given by the combination
\begin{equation}\label{LinearGeometry.ROANumber}
R=\left(\!\!\!\begin{array}{c}n+c-1 \\c-1 \end{array}\!\!\!\right)=\frac{\left(n+c-1\right)!}{n!\left(c-1\right)!}\,.
\end{equation}
Each instance of $\nu$ uniquely specifies a convex subset of $\mathscr{S}_0$ defined as\footnote{For $k=1$, the sum $\sum_{j=1}^{k-1}\nu_j$ is taken equal to $0$.}
\begin{align}\label{LinearGeometry.ROCSetDefinition}
&\mathscr{S}\left(\nu\right)=\bigl\{x|\,q_{k-1}<x_{i_k+1}<\cdots<x_{i_k+\nu_k}<q_k,\nonumber\\
&{\qquad\quad}~i_k=\textstyle\sum_{j=1}^{k-1}\nu_j,~k\in\left\{l=1,2,\ldots,c|\nu_l\neq0\right\}\bigr\}.
\end{align}
The convexity of this set is straightforward to show.

These subsets are disjoint and their union is equal to the state space $\mathscr{S}_0$ excluding a zero-measure set $\mathscr{B}_0$ containing the boundaries of the open sets $\mathscr{S}\left(\nu\right)$ in $\mathscr{S}_0$, i.e.,
\begin{equation*}
\bigcup_{\left\|\nu\right\|_1=n}\mathscr{S}\left(\nu\right)=\mathscr{S}_0\backslash\mathscr{B}_0.
\end{equation*}
Theorem~\ref{LinearGeometry.ExistenceUniquenessEquilibriumTheorem} in this section states that each subset $\mathscr{S}\left(\nu\right)$ of the state space $\mathscr{S}_0$ contains exactly one stable equilibrium of~\eqref{LinearGeometry.DeterministicEquationMotion} and that the energy function is convex over $\mathscr{S}\left(\nu\right)$, concluding that each $\mathscr{S}\left(\nu\right)$ is a ROA of the dynamical system~\eqref{LinearGeometry.DeterministicEquationMotion}.
\begin{theorem}\label{LinearGeometry.ExistenceUniquenessEquilibriumTheorem}
For any constant control $\mathbf{u}\left(t\right)=u_\mathrm{ss}$ with positive components, the energy function $V\left(\,\cdot\,,u_\mathrm{ss}\right)$ is strictly convex over each convex subset $\mathscr{S}\left(\nu\right)$ (with $\nu$ satisfying $\left\|\nu\right\|_1=n$), the dynamical system~\eqref{LinearGeometry.DeterministicEquationMotion} has exactly one equilibrium in~$\mathscr{S}\left(\nu\right)$, and that equilibrium is stable.
\end{theorem}
\begin{proof}
See Appendix.
\end{proof}

Formation of a pattern $\mathcal{P}$ at time $t$ is confirmed if the state vector $\mathbf{x}\left(t\right)$ is in the subset $\mathscr{P}\left(\mathcal{P}\right)$ of the state space $\mathscr{S}_0$. On the other hand, the definition of $\mathscr{P}\left(\mathcal{P}\right)$ and the structure of the ROAs imply that each $\mathscr{P}\left(\mathcal{P}\right)$ is entirely inside a single ROA. That specific ROA is characterized as follows. Assume that the pattern $\mathcal{P}$ has $\nu_k\left(\mathcal{P}\right)$ particles in the interval $\left(q_{k-1},q_k\right)$ and let $\nu\left(\mathcal{P}\right)$ denote a vector in $\mathbb{N}_0^c$ containing the integers $\nu_k\left(\mathcal{P}\right)$, $k=1,2,\ldots,c$. Then $\mathscr{P}\left(\mathcal{P}\right)$ is a subset of $\mathscr{S}\left(\nu\left(\mathcal{P}\right)\right)$ as defined in~\eqref{LinearGeometry.ROCSetDefinition}. Thus, to form a pattern $\mathcal{P}$, it is necessary to first bring the state vector inside the ROA $\mathscr{S}\left(\nu\left(\mathcal{P}\right)\right)$. For simplicity of notation in the rest of the paper, $\mathscr{S}\left(\nu\left(\mathcal{P}\right)\right)$ is abbreviated into $\mathscr{S}\left(\mathcal{P}\right)$ to represent the ROA containing the pattern $\mathcal{P}$.
%
%
\subsection{Optimal Static Control}\label{LinearGeometry.StaticControlSection}
Suppose that a dynamic control has been applied to the stochastic system~\eqref{LinearGeometry.StochasticEquationMotion} during the time interval $t\in\left[0,t_d\right)$ to bring its state inside the ROA $\mathscr{S}\left(\mathcal{P}_d\right)$ that contains the desired pattern $\mathcal{P}_d$. In Section~\ref{LinearGeometry.DynamicControlSection}, it is shown how to design such a dynamic control to maximize the probability of hitting the target set $\mathscr{S}\left(\mathcal{P}_d\right)$. At $t=t_d$, the constant static control $u_\mathrm{ss}$ is applied to the system and the system gradually reaches the steady-state as $t\to+\infty$. In the steady-state regime, the probability of forming the desired pattern remains constant, i.e., the event of $\mathbf{x}\left(t\right)\in\mathscr{P}\left(\mathcal{P}_d\right)$ has a constant probability. The objective of the static control is to maximize this constant probability assuming that at $t=t_d$ the state vector is inside the desired ROA, i.e., $\mathbf{x}\left(t_d\right)\in\mathscr{S}\left(\mathcal{P}_d\right)$. This goal is mathematically represented by the optimization~problem
\begin{equation}\label{LinearGeometry.OptimalStaticControl}
\max_{u_\mathrm{ss}\in\mathcal{U}_\mathrm{ss}}\lim_{t\to+\infty}\Pr\left\{\mathbf{x}\left(t\right)
\in\mathscr{P}\left(\mathcal{P}_d\right)|\,\mathbf{x}\left(t_d\right)\in\mathscr{S}\left(\mathcal{P}_d\right)\right\}.
\end{equation}

In practice, the system can get arbitrarily close to the steady-state within a bounded but long enough settling time $t_f-t_d$. Since the problem statement in Section~\ref{LinearGeometry.ModelSection} does not constrain the final time $t_f$, this quantity can be chosen sufficiently large to ensure that the conditional probability
\begin{equation*}
\Pr\left\{\mathbf{x}\left(t_f\right)\in\mathscr{P}\left(\mathcal{P}_d\right)|\,\mathbf{x}\left(t_d\right)
\in\mathscr{S}\left(\mathcal{P}_d\right)\right\}
\end{equation*}
is close enough to its final value in~\eqref{LinearGeometry.OptimalStaticControl}. Under this value of the final time, the static control that solves the optimization problem~\eqref{LinearGeometry.OptimalStaticControl}, nearly maximizes this conditional probability.

Here and in the rest of this paper, the static control $u_\mathrm{ss}$ is chosen from the control set $\mathcal{U}_\mathrm{ss}$ defined as a subset of $\mathbb{R}^{c+1}$ with positive control charges at two end points $q_0$ and $q_c$, and nonnegative charges for the rest of the electrodes, so that
\begin{equation}\label{LinearGeometry.SteadyStateControlSet}
\mathcal{U}_\mathrm{ss}=\left\{u|u_0>0,u_1\geqslant0,\ldots,u_{c-1}\geqslant0,u_c>0\right\}.
\end{equation}
Under this assumption, Theorem~\ref{LinearGeometry.ExistenceUniquenessEquilibriumTheorem} is applied to the ROAs of the system. Notice that in the statement of Theorem~\ref{LinearGeometry.ExistenceUniquenessEquilibriumTheorem}, all control charges are assumed positive, while~\eqref{LinearGeometry.SteadyStateControlSet} allows some electrodes to be inactive with zero charges. This provides more flexibility to the control vector without jeopardizing the use of Theorem~\ref{LinearGeometry.ExistenceUniquenessEquilibriumTheorem}: when some electrodes are inactive, still this theorem is applied, albeit to a system with smaller number of electrodes (with a control vector of smaller dimension). In~\eqref{LinearGeometry.SteadyStateControlSet}, only the electrodes at two end points $q_0$ and $q_c$ are constrained to be active to ensure that the particles cannot escape the line segment.

For $t\geqslant{t_d}$, define $\rho\left(x,t\right)$ as the conditional probability density function of $\mathbf{x}\left(t\right)$ given $\mathbf{x}\left(t_d\right)\in\mathscr{S}\left(\mathcal{P}_d\right)$. The evolution of this function for $t>{t_d}$ is governed by the Fokker-Planck equation
\begin{equation*}
\frac{{\partial}\rho}{\partial{t}}\left(x,t\right)=\nabla_x\cdot\left({\nabla_x}V\left(x,u_\mathrm{ss}\right)\rho\left(x,t\right)
+\frac{1}{2}\;\!\sigma^2{\nabla_x}\rho\left(x,t\right)\right).
\end{equation*}
The definition of $\rho$ implies that at $t=t_d$ this function is identically $0$ for every $x\notin\mathscr{S}\left(\mathcal{P}_d\right)$. Since the state vector is almost surely inside $\mathscr{S}\left(\mathcal{P}_d\right)$ at $t=t_d$ and almost surely cannot leave this ROA, it stays in $\mathscr{S}\left(\mathcal{P}_d\right)$ for every $t\geqslant{t_d}$ with probability $1$. This implies  $\rho\left(x,t\right)=0$ for every $x\notin\mathscr{S}\left(\mathcal{P}_d\right)$ and every $t\geqslant{t_d}$. Based on this analysis, the steady-state solution of this Fokker-Planck equation is given by~\cite{BOOK.Soize.94}
\begin{equation}\label{LinearGeometry.FokkerPlanckSteadyState}
\rho\left(x,+\infty\right)=\frac{\exp\left(-2\sigma^{-2}V\left(x,u_\mathrm{ss}\right)\right)}
{\displaystyle\int_{\mathscr{S}\left(\mathcal{P}_d\right)}\exp\left(-2\sigma^{-2}V\left(\xi,u_\mathrm{ss}\right)\right)d\xi}
\end{equation}
for $x\in\mathscr{S}\left(\mathcal{P}_d\right)$ and by $\rho\left(x,+\infty\right)=0$ for $x\notin\mathscr{S}\left(\mathcal{P}_d\right)$. Note that the normalizing factor in the denominator is an integral over $\mathscr{S}\left(\mathcal{P}_d\right)$ rather than the entire state space following the fact that the conditional probability distribution is identically~$0$ outside this ROA.

Using the conditional density function~\eqref{LinearGeometry.FokkerPlanckSteadyState}, the conditional probability
\begin{equation}\label{LinearGeometry.PinfDefinition}
P_\mathrm{ss}\left(u_\mathrm{ss}\right)\triangleq\lim_{t\to+\infty}\Pr\left\{\mathbf{x}\left(t\right)
\in\mathscr{P}\left(\mathcal{P}_d\right)|\,\mathbf{x}\left(t_d\right)
\in\mathscr{S}\left(\mathcal{P}_d\right)\right\}
\end{equation}
is expressed as
\begin{equation}\label{LinearGeometry.PinfExplititForm}
P_\mathrm{ss}\left(u_\mathrm{ss}\right)=\frac{\displaystyle\int_{\mathscr{P}\left(\mathcal{P}_d\right)}
\exp\left(-2\sigma^{-2}V\left(\xi,u_\mathrm{ss}\right)\right)d\xi}
{\displaystyle\int_{\mathscr{S}\left(\mathcal{P}_d\right)}\exp\left(-2\sigma^{-2}V\left(\xi,u_\mathrm{ss}\right)\right)d\xi}\,.
\end{equation}
The condition $P_\mathrm{ss}\left(u_\mathrm{ss}\right)\leqslant1$ on probabilities is reflected in this expression by the fact that $\mathscr{P}\left(\mathcal{P}_d\right)\subset\mathscr{S}\left(\mathcal{P}_d\right)$. Using the explicit form~\eqref{LinearGeometry.PinfExplititForm}, the optimal static control $u_\mathrm{ss}^*$ is obtained from the optimization problem
\begin{equation}\label{LinearGeometry.StaticControlOptimizationProblem}
u_\mathrm{ss}^*\in\arg\max_{u_\mathrm{ss}\in\mathcal{U}_\mathrm{ss}}P_\mathrm{ss}\left(u_\mathrm{ss}\right).
\end{equation}

In this optimization problem, it is computationally expensive to determine $P_\mathrm{ss}\left(u_\mathrm{ss}\right)$ from~\eqref{LinearGeometry.PinfExplititForm}, since numerical approximation of this expression requires computation of the energy function $V\left(\xi,u_\mathrm{ss}\right)$ at a large number of points. The computational complexity can be significantly reduced by \textit{saddle point approximation}~\cite{BOOK.SeriesApproximationMethods.94,BOOK.SaddlePoint.95} of the integrals in~\eqref{LinearGeometry.PinfExplititForm}. This approximation relies on the fact that the negative exponential integrands in~\eqref{LinearGeometry.PinfExplititForm} take their significant values in the ROA $\mathscr{S}\left(\mathcal{P}_d\right)$ only around the unique minimizer $x_\mathrm{ss}\left(u_\mathrm{ss}\right)$ of the energy function. Hence, without significant loss of accuracy, $V\left(\xi,u_\mathrm{ss}\right)$ can be replaced with a simpler function that approximates it well only around $x_\mathrm{ss}\left(u_\mathrm{ss}\right)$. A reasonable choice for such an approximation is the truncated Taylor series
\begin{equation*}
V\left(\xi,u_\mathrm{ss}\right)\simeq{V}\left(x_\mathrm{ss},u_\mathrm{ss}\right)
+\frac{1}{2}\left(\xi-x_\mathrm{ss}\right)^\mathrm{T}\!H\left(x_\mathrm{ss},u_\mathrm{ss}\right)\left(\xi-x_\mathrm{ss}\right)
\end{equation*}
in which the dependence of $x_\mathrm{ss}$ on $u_\mathrm{ss}$ is not explicitly shown for the sake of simplicity.

Substituting this approximate expression into~\eqref{LinearGeometry.PinfExplititForm} and multiplying both its numerator and denominator by an appropriate constant, $P_\mathrm{ss}\left(u_\mathrm{ss}\right)$ is approximated by
\begin{equation}\label{LinearGeometry.Pinfapproximation}
P_\mathrm{ss}\left(u_\mathrm{ss}\right)\simeq\tilde{P}_\mathrm{ss}\left(x_\mathrm{ss}\left(u_\mathrm{ss}\right),u_\mathrm{ss}\right)
\end{equation}
in terms of the mapping $\tilde{P}_\mathrm{ss}:\mathbb{R}^n\times\mathbb{R}^{c+1}\to\mathbb{R}$ defined as
\begin{equation}\label{LinearGeometry.PtildeInfDefinition}
\tilde{P}_\mathrm{ss}\left(x,u\right)=\frac{\displaystyle\int_{\mathscr{P}\left(\mathcal{P}_d\right)}
\Phi\left(\xi;x,\textstyle\frac{1}{2}\,\sigma^2H^{-1}\left(x,u\right)\right)d\xi}
{\displaystyle\int_{\mathscr{S}\left(\mathcal{P}_d\right)}\Phi\left(\xi;x,\textstyle\frac{1}{2}\,\sigma^2H^{-1}\left(x,u\right)\right)d\xi}\,.
\end{equation}
Here, $\Phi\left(\,\cdot\,;m,\Sigma\right)$ denotes a multivariate normal distribution with the mean vector $m$ and the covariance matrix $\Sigma$. Using the payoff function~\eqref{LinearGeometry.Pinfapproximation}-\eqref{LinearGeometry.PtildeInfDefinition}, the optimization problem~\eqref{LinearGeometry.StaticControlOptimizationProblem} can be reformulated as the constrained optimization problem
\begin{subequations}\label{LinearGeometry.GaussianApproximationOptimizationProblem}
\begin{align}
\max_{\left(x_\mathrm{ss},u_\mathrm{ss}\right)\in\mathscr{S}\left(\mathcal{P}_d\right)\times\mathcal{U}_\mathrm{ss}}&{\;\;}
\tilde{P}_\mathrm{ss}\left(x_\mathrm{ss},u_\mathrm{ss}\right)\\
\mathrm{s.t.}\qquad\quad\,\,&{\!\!\!\!\!\!\!\!\!\!\!\!}-\nabla_xV\left(x_\mathrm{ss},u_\mathrm{ss}\right)=\mathbf{0}.
\label{LinearGeometry.GaussianApproximationOptimizationProblemB}
\end{align}
\end{subequations}

It is shown next that the approximate formula~\eqref{LinearGeometry.PtildeInfDefinition} can be directly derived from a linearized model around the unique stable equilibrium of the ROA $\mathscr{S}\left(\mathcal{P}_d\right)$. Such a linear model is later used to provide an intuitive explanation for the static control design. Let $u_\mathrm{ss}\in\mathcal{U}_\mathrm{ss}$ be a fixed control and assume that $x_\mathrm{ss}$ is its associated equilibrium in $\mathscr{S}\left(\mathcal{P}_d\right)$. If the disturbance strength $\sigma$ is small compared to the norm of the Hessian matrix~$H\left(x_\mathrm{ss},u_\mathrm{ss}\right)$, the dynamical system~\eqref{LinearGeometry.StochasticEquationMotion} can be linearized around $\left(x_\mathrm{ss},u_\mathrm{ss}\right)$ to approximate the state vector as
\begin{equation}\label{LinearGeometry.DeterministicStochasticTerms}
\mathbf{x}\left(t\right)\simeq{x_\mathrm{ss}}+\mathbf{\delta{x}}\left(t\right)
\end{equation}
in which the small deviation $\mathbf{\delta{x}}\left(t\right)$ is the solution of the linear stochastic differential equation
\begin{equation}\label{LinearGeometry.LinearizedDynamicalSystem}
d\mathbf{\delta{x}}\left(t\right)=-H\left(x_\mathrm{ss},u_\mathrm{ss}\right)\mathbf{\delta{x}}\left(t\right)dt+{\sigma}d\mathbf{w}\left(t\right).
\end{equation}

The linearity of this equation indicates that $\mathbf{\delta{x}}\left(t\right)$ is a zero-mean Gaussian random vector with the steady-state covariance matrix
\begin{equation}\label{LinearGeometry.SteadyStateCovariance}
\Sigma_\mathrm{ss}=\frac{1}{2}\,\sigma^2H^{-1}\left(x_\mathrm{ss},u_\mathrm{ss}\right)
\end{equation}
that solves the algebraic Lyapunov equation~\cite{BOOK.GajicQureshi.95}
\begin{equation*}
-H\left(x_\mathrm{ss},u_\mathrm{ss}\right)\Sigma_\mathrm{ss}-\Sigma_\mathrm{ss}H\left(x_\mathrm{ss},u_\mathrm{ss}\right)+\sigma^2I=0.
\end{equation*}
It is concluded that $\mathbf{x}\left(t\right)$ is approximately a Gaussian random vector with the mean vector $x_\mathrm{ss}$ and the covariance matrix $\Sigma_\mathrm{ss}$. For this approximation, $\tilde{P}_\mathrm{ss}\left(x_\mathrm{ss},u_\mathrm{ss}\right)$ defined in~\eqref{LinearGeometry.PtildeInfDefinition} represents the conditional probability on the right-hand side of~\eqref{LinearGeometry.PinfDefinition}.

Based on the linear model~\eqref{LinearGeometry.DeterministicStochasticTerms}-\eqref{LinearGeometry.LinearizedDynamicalSystem}, an approximate design method is introduced below to provide further intuition on the optimal static control. To facilitate the discussion, $P_\mathrm{ss}\left(u_\mathrm{ss}\right)$ is expressed as the ratio
\begin{equation}\label{LinearGeometry.PtildeAlternative}
P_\mathrm{ss}\left(u_\mathrm{ss}\right)=\frac{A\left(u_\mathrm{ss}\right)}{A\left(u_\mathrm{ss}\right)+B\left(u_\mathrm{ss}\right)}
\end{equation}
with $A\left(u_\mathrm{ss}\right)$ and $B\left(u_\mathrm{ss}\right)$ defined as
%
\begin{align*}
A\left(u_\mathrm{ss}\right)&=\int_{\mathscr{P}\left(\mathcal{P}_d\right)}\exp\left(-2\sigma^{-2}V\left(\xi,u_\mathrm{ss}\right)\right)d\xi\\
B\left(u_\mathrm{ss}\right)&=\int_{\mathscr{S}\left(\mathcal{P}_d\right)\backslash\mathscr{P}\left(\mathcal{P}_d\right)}
\exp\left(-2\sigma^{-2}V\left(\xi,u_\mathrm{ss}\right)\right)d\xi.
\end{align*}
%

To achieve the maximum of $P_\mathrm{ss}$, the value $A\left(u_\mathrm{ss}\right)$ must be kept as large as possible compared to $B\left(u_\mathrm{ss}\right)$. This requires to shift the concentration of probability towards the central point of the desired set $\mathscr{P}\left(\mathcal{P}_d\right)$ and push it away from the forbidden set $\mathscr{S}\left(\mathcal{P}_d\right)\backslash\mathscr{P}\left(\mathcal{P}_d\right)$. Let $\xi_d\in\mathbb{R}^n$ be a vector containing the mid points of the intervals $\mathscr{I}_k$ corresponding to $1$'s in the desired pattern $\mathcal{P}_d$. Then the problem is to keep the state vector $\mathbf{x}\left(t\right)$ at steady-state as close as possible to $\xi_d$ with respect to some appropriate norm.

Since the desired set $\mathscr{P}\left(\mathcal{P}_d\right)$ is a hypercube, a reasonable choice of the distance measure is
\begin{equation*}
\lim_{t\to+\infty}\mathrm{E}\left[\,\left\|\mathbf{x}\left(t\right)-\xi_d\right\|_\infty\right]
\simeq\lim_{t\to+\infty}\mathrm{E}\left[\,\left\|x_\mathrm{ss}-\xi_d+\mathbf{{\delta}x}\left(t\right)\right\|_\infty\right].
\end{equation*}
However, this measure is not mathematically tractable unless the disturbance power $\frac{1}{2}\:\!\sigma^2$ is small enough to justify the approximation $\mathbf{{\delta}x}\left(t\right)\simeq0$. In this case, an approximation for the optimal static control is given by the optimization problem
\begin{align*}
\min_{\left(x_\mathrm{ss},u_\mathrm{ss}\right)\in\mathscr{S}\left(\mathcal{P}_d\right)\times\mathcal{U}_\mathrm{ss}}&
{\;\;}\left\|x_\mathrm{ss}-\xi_d\right\|_\infty\\
\mathrm{s.t.}\qquad\quad\,\,&{\!\!\!\!\!\!\!\!\!\!\!\!}-\nabla_xV\left(x_\mathrm{ss},u_\mathrm{ss}\right)=\mathbf{0}.
\end{align*}
Such a suboptimal static control places the stable equilibrium $x_\mathrm{ss}$ as close as possible to $\xi_d$ that represents the desired pattern. When $\frac{1}{2}\:\!\sigma^2$ is not negligible, one can alternatively adopt the mean squared distance measure
\begin{equation*}
\lim_{t\to+\infty}\mathrm{E}\left[\,\left\|\mathbf{x}\left(t\right)-\xi_d\right\|^2\right]\simeq\left\|x_\mathrm{ss}-\xi_d\right\|^2
+\mathrm{tr}\left\{\Sigma_\mathrm{ss}\right\}
\end{equation*}
which leads to the optimization problem
\begin{align*}
\min_{\left(x_\mathrm{ss},u_\mathrm{ss}\right)\in\mathscr{S}\left(\mathcal{P}_d\right)\times\mathcal{U}_\mathrm{ss}}&
{\;\;}\left\|x_\mathrm{ss}-\xi_d\right\|^2+\textstyle\frac{1}{2}\,\sigma^2\mathrm{tr}\left\{H^{-1}\left(x_\mathrm{ss},u_\mathrm{ss}\right)\right\}\\
\mathrm{s.t.}\qquad\quad\,\,&{\!\!\!\!\!\!\!\!\!\!\!\!}-\nabla_xV\left(x_\mathrm{ss},u_\mathrm{ss}\right)=\mathbf{0}.
\end{align*}

It is reasonable at this point to ask whether the optimization problem~\eqref{LinearGeometry.StaticControlOptimizationProblem} has always a bounded solution or it is possible for the optimal static control to be unbounded. An informal treatment of this problem comes below. Let $u_\mathrm{ss}$ be a constant control in $\mathcal{U}_\mathrm{ss}$ with a stable equilibrium $x_\mathrm{ss}\left(u_\mathrm{ss}\right)\in\mathscr{S}\left(\mathcal{P}_d\right)$. For every $\xi\in\mathscr{S}\left(\mathcal{P}_d\right)$, the value $V\left(\xi,u_\mathrm{ss}\right)$ of the energy function tends to $+\infty$ as $\left\|u_\mathrm{ss}\right\|\to+\infty$ since the components of $u_\mathrm{ss}$ are nonnegative. Under this limit, the ratio~\eqref{LinearGeometry.PtildeAlternative} tends to either $0$ or $1$, depending on which of the sets $\mathscr{S}\left(\mathcal{P}_d\right)\backslash\mathscr{P}\left(\mathcal{P}_d\right)$ or $\mathscr{P}\left(\mathcal{P}_d\right)$ contain $x_\mathrm{ss}\left(u_\mathrm{ss}\right)$. In particular, the limiting value of $P_\mathrm{ss}\left(u_\mathrm{ss}\right)$ is explicitly given by
\begin{equation*}
\lim_{\left\|u_\mathrm{ss}\right\|\to+\infty}P_\mathrm{ss}\left(u_\mathrm{ss}\right)=\left\{\!\!
\begin{array}{ll}
1 & x_\mathrm{ss}\left(u_\mathrm{ss}\right)\in\mathscr{P}\left(\mathcal{P}_d\right)\vspace{0.05in}\\
0&x_\mathrm{ss}\left(u_\mathrm{ss}\right)\in\mathscr{S}\left(\mathcal{P}_d\right)\backslash\mathscr{P}\left(\mathcal{P}_d\right).
\end{array}
\right.
\end{equation*}
This implies that $u_\mathrm{ss}^*$ in~\eqref{LinearGeometry.StaticControlOptimizationProblem} can be unbounded if under this control the equilibrium $x_\mathrm{ss}\left(u_\mathrm{ss}^*\right)$ remains inside $\mathscr{P}\left(\mathcal{P}_d\right)$.

Assume that the desired pattern $\mathcal{P}_d$ at least in one of the intervals $\left(q_{k-1},q_k\right)$ has two or more particles. Any unbounded control, necessarily squeezes these particles together, and two particles closer than a grid length cannot form a valid pattern, i.e., the probability of forming $\mathcal{P}_d$ under an unbounded control is identically $0$. It is concluded that the optimal control $u_\mathrm{ss}^*$ can be potentially unbounded only for the sparse patterns with at most one particle in each interval $\left(q_{k-1},q_k\right)$. Thus, except for such sparse patterns (that are not of much practical interest), it is not necessary to impose an upper bound on the control set $\mathcal{U}_\mathrm{ss}$ to secure a well defined solution for the optimization problem~\eqref{LinearGeometry.StaticControlOptimizationProblem}.

In theory, an unbounded control can achieve the probability $P_\mathrm{ss}\left(u_\mathrm{ss}^*\right)=1$ for a sparse pattern. In this degenerate case, the number of particles is smaller than the number of controls ($n<c+1$) so that the algebraic equation $-\nabla_xV\left(\xi_d;u_\mathrm{ss}\right)=\mathbf{0}$ can have multiple solutions for $u_\mathrm{ss}$, including a solution with an infinite magnitude. Such an unbounded control leads to a Hessian matrix $H\left(\xi_d,u_\mathrm{ss}\right)$ with an unbounded norm, which in turn, results in a zero covariance matrix according to~\eqref{LinearGeometry.SteadyStateCovariance}.
%
%
\subsubsection{Numerical Computation of a Stable Equilibrium}\label{LinearGeometry.NumericalComputationEquilibriumSection}
Two numerical techniques are proposed here to compute the stable equilibrium inside each ROA. The first method relies on the numerical solution of the ordinary differential equation~\eqref{LinearGeometry.DeterministicEquationMotion}. Starting from any arbitrary initial state inside the desired ROA, the solution to~\eqref{LinearGeometry.DeterministicEquationMotion} asymptotically approaches the unique stable equilibrium of that ROA. Although the exact equilibrium is approached only as $t\to+\infty$, after a bounded but long enough time, the solution of~\eqref{LinearGeometry.DeterministicEquationMotion} will be close enough to the equilibrium to provide an acceptable approximation for it.

The second method makes use of the proof of Theorem~\ref{LinearGeometry.ExistenceUniquenessEquilibriumTheorem} (see Appendix). This method starts from an initial vector $x^0\in\mathscr{S}\left(\nu\right)$ and generates the sequence of vectors $x^1,x^2,x^3,\ldots$ from the recursive equation $x^{k+1}=g\bigl(x^k\bigr)$, $k=0,1,2,\ldots$ until the distance between two successive vectors drops below a given threshold. Then the last vector in this sequence is taken as an approximation for the equilibrium. At each step $k$, the $i$th component of $g\bigl(x^k\bigr)$ is computed by numerically solving (e.g. using Newton's method) the algebraic equation
\begin{equation*}
f_i\bigl(x_1^k,x_2^k,\ldots,x_{i-1}^k,y,x_{i+1}^k,\ldots,x_n^k,u_\mathrm{ss}\bigr)=0
\end{equation*}
for $y$, as explained in Theorem~\ref{LinearGeometry.ExistenceUniquenessEquilibriumTheorem}.
%
%
\subsubsection{Approximate Settling Time}\label{LinearGeometry.SettlingTimeSection}
Determining an estimate for the settling time $t_f-t_d$ is the last step to complete the design of the static control. This quantity closely depends on the second smallest eigenvalue (in absolute value) of the Fokker-Planck operator~\cite{ART.LiberzonBrockett.00}
\begin{equation*}
\mathcal{L}\left(\cdot\right)=\nabla_x\cdot\left({\nabla_x}V\left(x,u_\mathrm{ss}\right)\left(\cdot\right)
+\frac{1}{2}\;\!\sigma^2{\nabla_x}\left(\cdot\right)\right).
\end{equation*}
Note that the first eigenvalue of this operator (smallest in the absolute value) is $0$ associated with the steady-state solution of the Fokker-Planck equation. Let $l_2\left(\mathcal{L}\right)<0$ be the second smallest eigenvalue of $\mathcal{L}\left(\cdot\right)$ in the absolute value. Then a rule of thumb for computation of the settling time is given by
\begin{equation*}
t_f-t_d\simeq5\left|l_2\left(\mathcal{L}\right)\right|^{-1}.
\end{equation*}

Direct computation of $l_2\left(\mathcal{L}\right)$ is generally a difficult task. In the case of this paper, the settling time can be approximated using an alternative method. This approximation is based on the observation that the settling time $t_f-t_d$ is the time required for the state vector $\mathbf{x}\left(t\right)$ to reach a stationary regime starting from an initial state inside the ROA $\mathscr{S}\left(\mathcal{P}_d\right)$. The temporal evolution of the state vector is governed by the stochastic differential equation~\eqref{LinearGeometry.StochasticEquationMotion} under the constant control $\mathbf{u}\left(t\right)=u_\mathrm{ss}$. The sample trajectories of $\mathbf{x}\left(t\right)$ almost surely remain inside the same ROA $\mathscr{S}\left(\mathcal{P}_d\right)$ and tend towards the equilibrium $x_\mathrm{ss}$. The large energy gradient near the boundary (large repulsive forces between closely placed point charges) of the ROA strongly pushes the trajectories towards the equilibrium so that they rapidly move away from the boundary and spend most of their transition time near the equilibrium. This justifies linearizing the nonlinear dynamics~\eqref{LinearGeometry.StochasticEquationMotion} around the equilibrium $x_\mathrm{ss}$ and computing the settling time from the linear model~\eqref{LinearGeometry.LinearizedDynamicalSystem} rather than the original nonlinear model~\eqref{LinearGeometry.StochasticEquationMotion}. Then the approximate settling time is expressed in terms of the smallest eigenvalue $l_1\left(\cdot\right)$ of the positive definite Hessian matrix $H\left(x_\mathrm{ss},u_\mathrm{ss}\right)$ as
\begin{equation}\label{LinearGeometry.SettlingTimeRule}
t_f-t_d\simeq5\:\!l_1^{-1}\left(H\left(x_\mathrm{ss},u_\mathrm{ss}\right)\right).
\end{equation}
%
%
\subsection{Optimal Electrode Positions}\label{LinearGeometry.OptimalElectrodeLocationsSection}
In the procedure proposed in Section~\ref{LinearGeometry.StaticControlSection} for design of the static control, the positions $q_0,q_1,\ldots,q_c$ of the electrodes are assumed fixed and given in advance. However, the specific choice of these positions, that are directly involved in the shape of the energy function~\eqref{LinearGeometry.EnergyFunction}, can affect the performance of the static control for better or worse. To maximize the probability of forming a desired pattern, the electrode positions can be optimized simultaneously with the static control, of course within certain physical constraints.

As mentioned before, the electrode positions are integer multiples of the grid size $d_0$. Let $N_k$ be an integer quantifying
the distance between the electrodes $k$ and $k-1$ as
\begin{equation*}
q_k-q_{k-1}=d_0N_k.
\end{equation*}
Then, in terms of $N_1,N_2,\ldots,N_c$, the electrode positions are given by
\begin{equation}\label{LinearGeometry.ElectrodePosition}
q_k=q_0+d_0\sum_{j=1}^kN_j,\quad{k=1,2,\ldots,c}.
\end{equation}
Since the length $q_c-q_0$ of the line segment is fixed and includes exactly $N$ grid cells, the integers $N_1,N_2,\ldots,N_c$ must satisfy the equality constraint $N_1+N_2+\cdots+N_c=N$. Moreover, technology limitations do not allow to fabricate the electrodes closer than $d_0N_\mathrm{min}$, imposing the inequality constraints $N_1,N_2,\ldots,N_c\geqslant{N}_\mathrm{min}$ for some integer $N_\mathrm{min}$.

Substituting~\eqref{LinearGeometry.ElectrodePosition} into the energy function~\eqref{LinearGeometry.EnergyFunction}, the dependence of the probability~\eqref{LinearGeometry.PinfExplititForm} on the energy function~\eqref{LinearGeometry.EnergyFunction} results in an explicit expression $P_\mathrm{ss}\left(u_\mathrm{ss},N_1,N_2,\ldots,N_c\right)$ for this probability. Then the joint optimization of the static control and the electrode positions can be formulated as the mixed integer nonlinear program (MINLP)~\cite{ART.KesavanEtAl.04,BOOK.TawarmalaniSahinidis.02}
\begin{equation*}
\begin{array}{cl}
\displaystyle\max_{u_\mathrm{ss},N_1,N_2,\ldots,N_c} & P_\mathrm{ss}\left(u_\mathrm{ss},N_1,N_2,\ldots,N_c\right) \vspace{0.02in}\\
\mathrm{s.t.} & u_\mathrm{ss}\in\mathcal{U}_\mathrm{ss} \vspace{0.08in}\\
&N_1+N_2+\cdots+N_c=N \vspace{0.08in}\\
&N_1,N_2,\ldots,N_c\geqslant{N}_\mathrm{min}.
\end{array}
\end{equation*}
%
%
\subsection{Dynamic Control}\label{LinearGeometry.DynamicControlSection}
At the initial time $t=0$, the state vector $\mathbf{x}\left(0\right)$ is randomly distributed in the state space $\mathscr{S}_0$ according to some probability density function $p_0$. The dynamic control is an open-loop control applied to the system of particles during $\left[0,t_d\right)$ to bring the random initial state inside the desired ROA $\mathscr{S}\left(\mathcal{P}_d\right)$ with the highest probability. The dynamic control proposed in this paper consists of a sequence of static controls, each one designed through a procedure similar to Section~\ref{LinearGeometry.StaticControlSection}. The concept of a multistage control for directed self-assembly has been established in~\cite{ART.SolisEtAl.10B} and is illustrated in \Fig~\ref{LinearGeometry.DynamicControlStages}.
\begin{figure}[h]
\begin{center}
\includegraphics[scale=1]{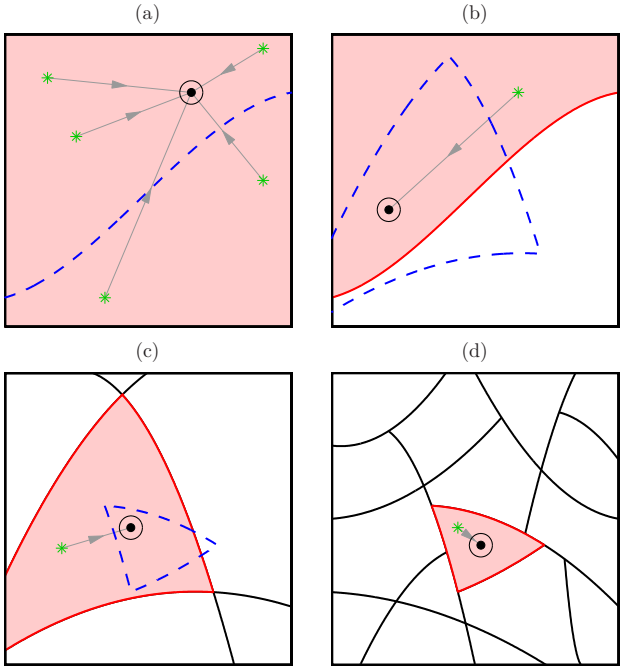}\\
\caption{Multistage dynamic control consisting of a sequence of static controls. The large squares represent the entire state space and the solid lines inside these squares show the boundary of the ROAs. In the first stage (a), the entire state space is a single ROA with a unique stable equilibrium marked by a small disk. The state vector can be initially any point in this single ROA (marked by the asterisks), which moves towards the stable equilibrium and eventually stays near this point (shown as a circular region) with a high probability. The stable equilibrium is designed to be inside the desired ROA (shaded region) of the next stage (b) so that the state vector will be inside this desired ROA at the end of the first stage. This procedure is repeated in transition from (b) to (c) and from (c) to (d). The static control in (d) eventually creates the desired pattern.}\label{LinearGeometry.DynamicControlStages}
\end{center}
\end{figure}

Consider a constant control vector $u_d^1$ whose first and last components are positive and the rest of its components are $0$ (only the electrodes at $q_0$ and $q_c$ are active). Under this control, all particles are distributed inside a single large interval $\left(q_0,q_c\right)$ and the system has a single ROA $\mathscr{S}_d^1$ covering the entire state space $\mathscr{S}_0$. Application of this constant control at $t=0$ drives the state vector towards the unique  equilibrium of $\mathscr{S}_d^1$ regardless of its initial value. After a settling time of $t_d^1$, the constant control $u_d^1$ is switched to another constant control $u_d^2$ with more than two positive components (e.g. the electrode in the middle is activated). Under this new control, the system has more than one ROA while a specific one of these multiple ROAs (denoted by $\mathscr{S}_d^2$) is the target set of the state vector at $t=t_d^1$. The optimal value of $u_d^1$ at the first stage is determined to maximize the probability of the state vector being inside this target ROA at $t=t_d^1$, i.e., $\mathbf{x}\left(t_d^1\right)\in\mathscr{S}_d^2$. This procedure is repeated in $D$ stages by activating more controls at each stage while getting closer to the final desired equilibrium. The target set $\mathscr{S}_d^{D+1}$ of the last stage is the ROA containing the desired pattern in the static control problem, i.e., $\mathscr{S}_d^{D+1}=\mathscr{S}\left(\mathcal{P}_d\right)$. At the end of the last stage of the dynamic control, the static control $u_\mathrm{ss}$ activates all electrodes to create the desired pattern.

Design of the target ROAs $\mathscr{S}_d^2,\mathscr{S}_d^3,\ldots,\mathscr{S}_d^D$ is explained below by an example using \Fig~\ref{LinearGeometry.FigDynamicControlExample}. In this figure, a system of $n=8$ particles with $c+1=5$ controls and  $N=16$ grid cells is illustrated. The dynamic control is designed with $D=2$ stages and the following sequence of controls: in the first stage, only the electrodes at $q_0$ and $q_4$ are active, while in the second and last stage the electrode at $q_2$ is turned on. The electrodes at $q_1$ and $q_3$ are simultaneously activated in the static control phase. According to this sequence, the controls $u_d^1$, $u_d^2$, and $u_\mathrm{ss}$ have the structure
\begin{align*}
u_d^1&=\left(\oplus,0,0,0,\oplus\right)\\
u_d^2&=\left(\oplus,0,\star,0,\oplus\right)\\
u_\mathrm{ss}&=\left(\oplus,\star,\star,\star,\oplus\right),
\end{align*}
where $\oplus$ and $\star$ denote positive and nonnegative components, respectively. The target ROA $\mathscr{S}_d^2$ is a subset of the state space with $5$ particles on the left side of~$q_2$ and $3$ particles on its right side, and is explicitly given~by
\begin{equation*}
\mathscr{S}_d^2=\left\{x|\,q_0<x_1<\cdots<x_5<q_2<x_6<x_7<x_8<q_4\right\}.
\end{equation*}
Similarly, $\mathscr{S}_d^3$ is a subset with $3$, $2$, $1$, and $2$ particles in the intervals $\left(q_0,q_1\right)$, $\left(q_1,q_2\right)$, $\left(q_2,q_3\right)$, and $\left(q_3,q_4\right)$, respectively.

Consider the family of piecewise constant controls
\begin{equation}\label{LinearGeometry.PiecewiseConstantControlFamily}
\mathbf{u}\left(t\right)=\left\{\!\!
\begin{array}{ll}
u_d^i\,, & t\in\left[t_d^{i-1},t_d^i\right),\quad{i=1,2,\ldots,D}\vspace{0.05in} \\
u_\mathrm{ss}\,, & t\in\left[t_f,+\infty\right),
\end{array}\right.
\end{equation}
where the constants $0=t_d^0<t_d^1<t_d^2<\cdots<t_d^D=t_d<t_f$ are the switching times of the control and $t_f$ is the final time. For $i=1,2,\ldots,D$, define the increasing sequence of control sets $\mathcal{U}_d^1\subset\mathcal{U}_d^2\subset\cdots\subset\mathcal{U}_d^D\subset\mathcal{U}_\mathrm{ss}$ with $\mathcal{U}_\mathrm{ss}$ given by~\eqref{LinearGeometry.SteadyStateControlSet}, $\mathcal{U}_d^1$ defined as
\begin{equation*}
\mathcal{U}_d^1=\left\{u|u_0>0,u_1=0,\ldots,u_{c-1}=0,u_c>0\right\},
\end{equation*}
and with the property that the controls in $\mathcal{U}_d^{i+1}$ can possess at least one more nonzero component than those in $\mathcal{U}_d^i$. It is assumed that at each stage $i=1,2,\ldots,D$, the constant control $u_d^i$ is in the control set $\mathcal{U}_d^i$, while the static control $u_\mathrm{ss}$ belongs to $\mathcal{U}_\mathrm{ss}$.

The goal is to obtain within the class of controls~\eqref{LinearGeometry.PiecewiseConstantControlFamily}, the one that solves the optimal control problem of Section~\ref{LinearGeometry.ModelSection}. The controls in this class consist of a static part applied for $t\geqslant{t_d}$ and a dynamic part during $0\leqslant{t}<t_d$. Optimization of the static part was discussed in Section~\ref{LinearGeometry.StaticControlSection} and the procedure for optimizing the dynamic part is presented below.

Consider the sequence $\mathscr{S}_d^2,\mathscr{S}_d^3,\ldots,\mathscr{S}_d^D$ of ROAs associated with the sequence $\mathcal{U}_d^1,\mathcal{U}_d^2,\ldots,\mathcal{U}_d^D$ of control sets, and for $i=1,2,\ldots,D$ define the conditional probabilities
\begin{equation}\label{LinearGeometry.PdiDefinition}
P_d^i\left(u_d^i\right)=\lim_{t\to+\infty}\Pr\left\{\mathbf{x}\left(t\right)\in\mathscr{S}_d^{i+1}\:\!|\:\!\mathbf{x}\left(t_d^{i-1}\right)
\in\mathscr{S}_d^i\right\},
\end{equation}
where $\mathbf{x}\left(t\right)$ is the state of~\eqref{LinearGeometry.StochasticEquationMotion} under the constant control $u_d^i$ applied at $t=t_d^{i-1}$. As discussed in Section~\ref{LinearGeometry.StaticControlSection}, these probabilities are explicitly expressed as
\begin{equation*}
P_d^i\left(u_d^i\right)=\frac{\displaystyle\int_{\mathscr{S}_d^{i+1}}\exp\left(-2\sigma^{-2}V\left(\xi,u_d^i\right)\right)d\xi}
{\displaystyle\int_{\mathscr{S}_d^i}\exp\left(-2\sigma^{-2}V\left(\xi,u_d^i\right)\right)d\xi}\,.
\end{equation*}
It is assumed that the time durations $t_d^i-t_d^{i-1}$ between the switching times are long enough for the system to reach the steady-state before application of a new segment of the control (as discussed in Section~\ref{LinearGeometry.SettlingTimeSection}). Under this condition, the following approximation holds:
\begin{equation*}
\Pr\left\{\mathbf{x}\left(t_d^i\right)\in\mathscr{S}_d^{i+1}\:\!|\:\!\mathbf{x}\left(t_d^{i-1}\right)
\in\mathscr{S}_d^i\right\}\simeq{P}_d^i\left(u_d^i\right).
\end{equation*}
At the initial time $t=0$, the state vector belongs to $\mathscr{S}_d^1=\mathscr{S}_0$ with probability $1$. Thus, the probability of formation of a desired pattern $\mathcal{P}_d$ under the control~\eqref{LinearGeometry.PiecewiseConstantControlFamily} at $t=t_f$ is given in terms of the conditional probabilities $P_\mathrm{ss}\left(u_\mathrm{ss}\right)$ and $P_d^i\left(u_d^i\right)$, $i=1,2,\ldots,D$ by the product
\begin{equation}\label{LinearGeometry.ContinuousProbabilitySuccess}
\Pr\left\{\mathbf{x}\left(t_f\right)\in\mathscr{P}\left(\mathcal{P}_d\right)\right\}\simeq{P}_\mathrm{ss}\left(u_\mathrm{ss}\right)
\prod_{i=1}^DP_d^i\left(u_d^i\right).
\end{equation}
Note that this approximation tends to exact as $t_f-t_d\to+\infty$ and $t_d^i-t_d^{i-1}\to+\infty$, $i=1,2,\ldots,D$.

To maximize the probability~\eqref{LinearGeometry.ContinuousProbabilitySuccess}, each multiplicative term must be maximized independently with respect to its argument. For the static control term, the optimization problem~\eqref{LinearGeometry.StaticControlOptimizationProblem} was already discussed in Section~\ref{LinearGeometry.StaticControlSection}. For the rest of the terms, the optimal controls $u_d^{i*}$, $i=1,2,\ldots,D$ are obtained by solving the optimization problems
\begin{equation}\label{LinearGeometry.DynamicControlOptimizationProblem}
u_d^{i*}\in\arg\max_{u_d^i\in\mathcal{U}_d^i}P_d^i\left(u_d^i\right).
\end{equation}
In terms of these optimal controls, the maximum probability of forming a desired pattern $\mathcal{P}_d$ at a large final time $t_f$ achieved by a control in the class of controls~\eqref{LinearGeometry.PiecewiseConstantControlFamily} is given by
\begin{equation}\label{LinearGeometry.MaximumPatternProbability}
\max\Pr\left\{\mathbf{x}\left(t_f\right)\in\mathscr{P}\left(\mathcal{P}_d\right)\right\}\simeq{P}_\mathrm{ss}\left(u_\mathrm{ss}^*\right)
\prod_{i=1}^DP_d^i\left(u_d^{i*}\right).
\end{equation}

The activation sequence of the electrodes is not unique and can be regarded as an additional optimization variable. For any specific activation sequence, the maximum probability to form a desired pattern is obtained from~\eqref{LinearGeometry.MaximumPatternProbability}. Then in a higher level of the optimization process, the maximum of these optimized probabilities is determined over all possible activation sequences. For the small size problem of \Fig~\ref{LinearGeometry.FigDynamicControlExample} with only $13$ possible activation sequences, this level of optimization can be performed by simply enumerating all sequences. For a problem of larger size, more advanced techniques can be developed based on the outer approximation~\cite{ART.KesavanEtAl.04} or branch and bound~\cite{BOOK.TawarmalaniSahinidis.02} methods.
%
%
\subsection{Numerical Results}
We applied the design procedure developed in this section to the example of \Fig~\ref{LinearGeometry.FigDynamicControlExample}. In this figure, self-assembly of $n=8$ particles using $c+1=5$ electrodes is considered along a line segment. The line segment is partitioned into $N=16$ cells and the distance between the electrodes is assumed to be $1$ unit of length. Throughout the design procedure and its following simulations the disturbance power is set at $\sigma=0.45$. The goal is to generate a desired pattern of
\begin{equation*}
\mathcal{P}_d=\left(0,1,1,1,0,0,1,1,0,0,1,0,0,1,0,1\right).
\end{equation*}

For design of both static and dynamic controls, the constrained optimization problem~\eqref{LinearGeometry.GaussianApproximationOptimizationProblem} was utilized to approximate the original problems~\eqref{LinearGeometry.StaticControlOptimizationProblem} and~\eqref{LinearGeometry.DynamicControlOptimizationProblem}. This constrained optimization problem, was converted to an unconstrained problem by solving the constraint~\eqref{LinearGeometry.GaussianApproximationOptimizationProblemB} for $x_\mathrm{ss}$ using the second numerical procedure proposed in Section~\ref{LinearGeometry.NumericalComputationEquilibriumSection}. The resulting unconstrained problem was solved using the \texttt{fminsearch} function of $\mbox{MATLAB}^\circledR$.

The designed optimal control consists of a static control and a two-stage dynamic control as illustrated in \Fig~\ref{LinearGeometry.FigDesignedContrrol}. The optimal values of the control vector are shown for the two stages of the dynamic control and for the static control in \Figs~\ref{LinearGeometry.FigDesignedContrrol}(a), \ref{LinearGeometry.FigDesignedContrrol}(b), and \ref{LinearGeometry.FigDesignedContrrol}(c), respectively. Further, the most likely patterns formed at the end of each stage are illustrated in these figures. The small boxes in these figures represent the locations of the electrodes, while the disks mark the optimal equilibrium of each stage. Using the approximation method of Section~\ref{LinearGeometry.SettlingTimeSection}, the switching times of the control were computed as $t_d^1=0.67$, $t_d^2=0.98$, and $t_f=1.13$. Also, the highest probability of success at each stage was obtained as $P_d^1\left(u_d^{1*}\right)\simeq1$, $P_d^2\left(u_d^{2*}\right)\simeq1$, and $P_\mathrm{ss}\left(u_\mathrm{ss}^*\right)=0.94$, which lead to the total probability of success
\begin{equation*}
P_d^1\left(u_d^{1*}\right)P_d^2\left(u_d^{1*}\right)P_\mathrm{ss}\left(u_\mathrm{ss}^*\right)\simeq0.94.
\end{equation*}
\begin{figure}[h]
\begin{center}
\includegraphics[scale=1]{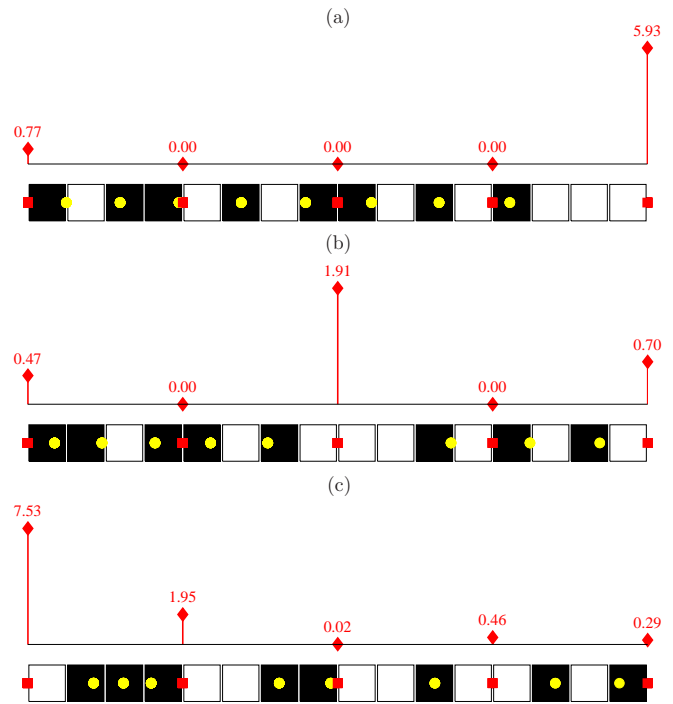}\\
\caption{Designed optimal control: (a) first stage of the dynamic control; (b) second stage of the dynamic control; (c) static control. The vertical lines with a number on their top represent the components of the control vector. The illustrated pattern at each stage represent the most likely pattern generated at the end of that stage. The small boxes mark the locations of the electrodes and the disks show the equilibrium of each stage. }\label{LinearGeometry.FigDesignedContrrol}
\end{center}
\end{figure}

Under this control, the deterministic dynamical system~\eqref{LinearGeometry.DeterministicEquationMotion} and its stochastic version~\eqref{LinearGeometry.StochasticEquationMotion} were numerically simulated. The results of these simulations are shown in \Fig~\ref{LinearGeometry.FigSiumulationResults} for both deterministic (thick line) and stochastic (thin line) models. The trajectory of the particles start at random positions (small boxes) and end at the optimal final equilibrium (small disks) that represents the desired pattern shown on the top. The vertical axis in this figure shows progress in time while the horizontal axis stands for the line segment on which the particles move. The heavy vertical lines represent the energy barriers created by the electrodes which cannot be crossed by the particles. Note that the specific sample path illustrated in \Fig~\ref{LinearGeometry.FigSiumulationResults} succeeds to create the desired pattern; however, not all sample paths of the stochastic system end up with the desired pattern. For example, the sample path of \Fig~\ref{LinearGeometry.FigSiumulationResults} could fail if the deviation marked by the dashed ellipse would occur shortly later. With the probability of success estimated as $0.94$, the sample paths fail to form the desired pattern at an average rate of $6\%$.
\begin{figure}[h]
\begin{center}
\includegraphics[scale=1]{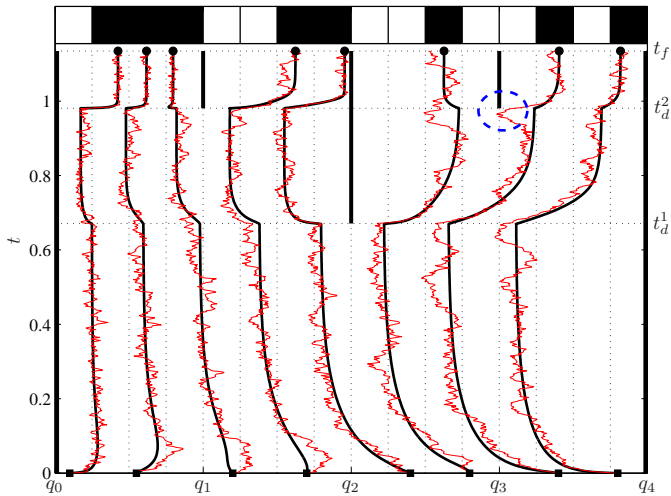}\\
\caption{Simulation results for deterministic (thick line) and stochastic (thin line) models. The vertical axis represents time while the horizontal axis stands for the line segment on which the particles move. The heavy vertical lines represent the energy barriers created by the electrodes which cannot be crossed by the particles. Not all sample paths of the stochastic system end up with the desired pattern; on average, $6\%$ of them fail to form this specific pattern. The sample path of \Fig~\ref{LinearGeometry.FigSiumulationResults} could fail if the deviation marked by the dashed ellipse would occur shortly later.}\label{LinearGeometry.FigSiumulationResults}
\end{center}
\end{figure}
%
%
\section{Control Design for a Discrete Ising Model}\label{LinearGeometry.DiscreteModelSection}
In the prior work on directed self-assembly, the system of particles has been described by a discrete Ising model~\cite{ART.Cipra.87} rather than the continuous model of this paper~\cite{ART.SolisEtAl.10A,ART.SolisEtAl.10B,ART.LakerveldEtAl.12}. In such a discrete model, the particles can occupy only a finite set of positions along the line---at most one particle in each position. As shown in \Fig~\ref{LinearGeometry.FigDynamicControlExample}, these positions are located at the centers of the intervals $\mathscr{I}_k$. In this model, the discrete state of the particles is a vector $\hat{\mathbf{x}}\left(t\right)$ in $\mathbb{R}^n$ taking its value in the discrete state space $\left\{\xi_1,\xi_2,\ldots,\xi_S\right\}\subset\mathbb{R}^n$ at any time $t\geqslant0$. Each vector $\xi_k$ in this discrete state space corresponds to a pattern $\mathcal{P}_k$ and its $n$ components represent the discrete locations of $n$ particles along the line. The number $S$ of the elements in the discrete state space is the same as the number of patterns that can be formed by placement of $n$ particles in $N>n$ positions and is given by the combination $S=\left(\substack{N\\{n}}\right)$. Equivalently, the discrete state of the particles at time $t$ can be described by an integer $z\left(t\right)\in\left\{1,2,\ldots,S\right\}$ such that $\hat{\mathbf{x}}\left(t\right)=\xi_{z\left(t\right)}$.

Similar to the continuous model~\eqref{LinearGeometry.StochasticEquationMotion}, the discrete state $z\left(t\right)$ is characterized by a stochastic process, namely by a continuous-time Markov chain~\cite{BOOK.GuoHernandezLerma.09}. For a continuous-time Markov chain, the evolution of probability is described by a \textit{master equation}, the analogue of the Fokker-Planck equation in the continuous model. For $i=1,2,\ldots,S$ define the probabilities
\begin{equation*}
\pi_i\left(t\right)=\Pr\left\{\hat{\mathbf{x}}\left(t\right)=\xi_i\right\}=\Pr\left\{z\left(t\right)=i\right\}.
\end{equation*}
These probabilities evolve in time according to the set of linear differential equations
\begin{equation*}
\dot{\pi}_i\left(t\right)=-\sum_{\substack{j=1\\j\neq{i}}}^{S}\lambda_{ji}\left(\mathbf{u}\left(t\right)\right)\pi_i\left(t\right)
+\sum_{\substack{j=1\\j\neq{i}}}^{S}\lambda_{ij}\left(\mathbf{u}\left(t\right)\right)\pi_j\left(t\right)
\end{equation*}
defined for $i=1,2,\ldots,S$. The initial state of these equations (the initial probability distribution) is assumed known and the nonnegative scalars $\lambda_{ij}\left(\mathbf{u}\left(t\right)\right)$, $i\neq{j}=1,2,\ldots,S$ describe the transition rates from pattern $j$ to pattern $i$. The contribution of the control to the system dynamics is reflected in the model through the dependence of the transition rates on the control vector $\mathbf{u}\left(t\right)$. By collecting the probabilities $\pi_i\left(t\right)$ in a single vector $\pi\left(t\right)=\left[\pi_1\left(t\right)\;\pi_2\left(t\right)\;\cdots\;\pi_S\left(t\right)\right]^\mathrm{T}$,~the master equation can be written in the compact form
\begin{equation}\label{LinearGeometry.MasterEquation}
\dot{\pi}\left(t\right)=\Lambda\left(\mathbf{u}\left(t\right)\right)\pi\left(t\right),
\end{equation}
where $\Lambda\left(\mathbf{u}\left(t\right)\right)$ is a $S\times{S}$ matrix with off-diagonal elements $\lambda_{ij}\left(\mathbf{u}\left(t\right)\right)$ and diagonal elements
\begin{equation*}
\lambda_{ii}\left(\mathbf{u}\left(t\right)\right)=-\sum_{\substack{j=1\\j\neq{i}}}^S\lambda_{ji}\left(\mathbf{u}\left(t\right)\right).
\end{equation*}

For any fixed $t\geqslant0$, the transition rate $\lambda_{ij}\left(\mathbf{u}\left(t\right)\right)$ from a pattern $j$ to another pattern $i$ exponentially decreases with the difference between the energy $V\left(\xi_j,\mathbf{u}\left(t\right)\right)$ of the pattern $j$ and
the energy barrier $E_{ij}\left(\mathbf{u}\left(t\right)\right)$ between the two patterns. The mapping  $\lambda_{ij}:\mathbb{R}^{c+1}\to\mathbb{R}^+$ that maps the instantaneous value of the control vector into the instantaneous value of the transition rate has the form of~\cite{ART.KangWeinberg.92}
\begin{equation*}
\lambda_{ij}\left(u\right)=\exp\left(-2\sigma^{-2}\left(E_{ij}\left(u\right)-V\left(\xi_j,u\right)\right)\right),\quad{i\neq{j}},
\end{equation*}
where $\frac{1}{2}\:\!\sigma^2={\kappa}k_BT$ is a constant increasing with the absolute temperature $T$, and the energy barrier satisfies the conditions
\begin{equation*}
E_{ij}\left(u\right)=E_{ji}\left(u\right)\geqslant{V}\left(\xi_j,u\right),\quad{u\in\mathcal{U}}.
\end{equation*}

For the explicit form of $E_{ij}\left(u\right)$, the reader is referred to~\cite{ART.LakerveldEtAl.12}. For the analysis of this paper, it is enough to know that the energy barrier between two patterns is positive, and it is of infinite magnitude if in transition from one pattern to another a particle has to jump over an active electrode. This latter property is caused by the unbounded level of energy at a point charge that represents an active electrode. Due to this property, the transition rates between two such patterns are identically $0$ which parallels the property observed in the continuous model that a particle cannot jump over an active electrode.

Because of this property, the central role of the ROAs in the continuous model has an analogue in the discrete model. Each ROA contains a certain subset of the patterns $\xi_1,\xi_2,\ldots,\xi_S$, and similar to the continuous model, the ROAs partition the discrete state space $\left\{\xi_1,\xi_2,\ldots,\xi_S\right\}$ into $R$ subsets ($R$ is the number of ROAs given by~\eqref{LinearGeometry.ROANumber}). The ROAs are marked by unbounded energy barriers encircling them, and such infinite energy barriers block the transition of a pattern inside a ROA to any pattern outside it. This implies that for any pair of patterns $i$ and $j$ inside two different ROAs, both othe transition rates $\lambda_{ij}\left(\mathbf{u}\left(t\right)\right)$ and $\lambda_{ji}\left(\mathbf{u}\left(t\right)\right)$ are identically $0$. As a result, the state-space equation~\eqref{LinearGeometry.MasterEquation} is decomposed into $R$ decoupled smaller state-space equations
\begin{equation*}
\dot{\pi}_i^\mathrm{ROA}\left(t\right)=\Lambda_i^\mathrm{ROA}\left(\mathbf{u}\left(t\right)\right)\
\pi_i^\mathrm{ROA}\left(t\right),\quad{i=1,2,\ldots,R},
\end{equation*}
where $\pi_i^\mathrm{ROA}\left(t\right)$ is a column vector containing the probabilities of the patterns belonging to the $i$th ROA and $\Lambda_i^\mathrm{ROA}\left(\mathbf{u}\left(t\right)\right)$ is its corresponding square block of $\Lambda\left(\mathbf{u}\left(t\right)\right)$. Equivalently, the Markov chain $z\left(t\right)$ is reducible into $R$ smaller Markov chains $z_1\left(t\right),z_2\left(t\right),\ldots,z_R\left(t\right)$ which are statistically independent conditioned on the initial value $z\left(0\right)$. Each of the $R$ decoupled state-space equations has a steady-state solution which follows a Gibbs probability distribution.

This property helps to determine a simple expression for the discrete counterpart of the conditional probability~\eqref{LinearGeometry.PinfDefinition}. Let $\mathcal{P}_d$ be a desired pattern inside the ROA $\mathscr{S}\left(\mathcal{P}_d\right)$ and assume that it is represented in the discrete state space by the vector $\xi_d$. Similar to~\eqref{LinearGeometry.PinfDefinition}, define the conditional probability
\begin{equation*}
\Pi_\mathrm{ss}\left(u_\mathrm{ss}\right)=\lim_{t\to+\infty}\Pr\left\{\hat{\mathbf{x}}\left(t\right)
=\xi_d\:\!|\;\!\hat{\mathbf{x}}\left(t_d\right)\in\mathscr{S}\left(\mathcal{P}_d\right)\right\}
\end{equation*}
under the constant control $\mathbf{u}\left(t\right)=u_\mathrm{ss}$ applied at $t=t_d$. Using the analysis above, this conditional probability can be explicitly expressed as
\begin{equation*}
\Pi_\mathrm{ss}\left(u_\mathrm{ss}\right)=\frac{\exp\left(-2\sigma^{-2}V\left(\xi_d,u_\mathrm{ss}\right)\right)}
{\displaystyle\sum_{\xi_k\in\mathscr{S}\left(\mathcal{P}_d\right)}^{}\exp\left(-2\sigma^{-2}V\left(\xi_k,u_\mathrm{ss}\right)\right)}\,.
\end{equation*}
Similarly, the discrete counterpart of~\eqref{LinearGeometry.PdiDefinition} is defined as
\begin{equation*}
\Pi_d^i\left(u_d^i\right)=\lim_{t\to+\infty}\Pr\left\{\hat{\mathbf{x}}\left(t\right)
\in\mathscr{S}_d^{i+1}\:\!|\;\!\hat{\mathbf{x}}\left(t_d^{i-1}\right)\in\mathscr{S}_d^i\right\}
\end{equation*}
under the constant control $\mathbf{u}\left(t\right)=u_d^i$ for $t\geqslant{t}_d^{i-1}$, and is explicitly expressed as
\begin{equation*}
\Pi_d^i\left(u_d^i\right)=\frac{\displaystyle\sum_{\xi_k\in\mathscr{S}_d^{i+1}}\exp\left(-2\sigma^{-2}V\left(\xi_k,u_d^i\right)\right)}
{\displaystyle\sum_{\xi_k\in\mathscr{S}_d^i}^{}\exp\left(-2\sigma^{-2}V\left(\xi_k,u_d^i\right)\right)}\,.
\end{equation*}
Finally, the probability of successful formation of a pattern $\mathcal{P}_d$ at time $t=t_f$ under the control~\eqref{LinearGeometry.PiecewiseConstantControlFamily} is approximately given~as
\begin{equation}\label{LinearGeometry.DiscreteProbabilitySuccess}
\Pr\left\{\hat{\mathbf{x}}\left(t_f\right)=\xi_d\right\}
\simeq\Pi_\mathrm{ss}\left(u_\mathrm{ss}\right)\prod_{i=1}^D\Pi_d^i\left(u_d^i\right).
\end{equation}

Evidently, the procedure of control design for the discrete model parallels the one proposed for the continuous model: it is enough to maximize the payoff function~\eqref{LinearGeometry.DiscreteProbabilitySuccess} instead of~\eqref{LinearGeometry.ContinuousProbabilitySuccess}. Further, each of these payoff functions closely approximates the other one as the sums in~\eqref{LinearGeometry.DiscreteProbabilitySuccess} are discrete approximations of the integrals in~\eqref{LinearGeometry.ContinuousProbabilitySuccess}. The only major difference in the control design procedure is in computation of the practical values of the settling times $t_f-t_d$ and $t_d^{i+1}-t_d^i$. For the continuous model, these quantities are determined in terms of the eigenvalues of the Fokker-Planck operator as explained in Section~\ref{LinearGeometry.SettlingTimeSection}. For the discrete model, the settling time $t_f-t_d$ is determined in terms of the smallest nonzero eigenvalue of the matrix $\Lambda_d^\mathrm{ROA}\left(u_\mathrm{ss}\right)$ (smallest in the absolute value), where the subscript $d$ refers to the ROA containing the desired pattern.
%
%
\section{Conclusion}
Directed self-assembly of charged particles along line segments has been considered. In the assembly process, a number of particles move in one dimension along a line segment under the repulsive forces experienced from interactions with other particles, and the process is directed towards formation of a desired pattern by external forces applied from charged electrodes located at fixed points in the line segment. The potentials of these electrodes are precisely controlled in time so that the formation of a desired pattern is secured with the highest probability despite the inherent uncertainty in the initial position and the dynamical behaviors of the particles. A challenging aspect of such a control is that the actual positions of the particles are not measurable during the assembly process. Two models have been proposed to describe the uncertain dynamics of the particles. The first model which is mathematically more tractable consists of a set of nonlinear stochastic differential equations and is suitable for larger particles of micrometer scale. The second model is a discrete Ising model consisting of a continuous-time Markov chain and is more accurate for nanometer scale particles. A class of piecewise constant controls has been proposed for these models and the optimal values of the constant pieces have been determined as the solutions to certain optimization problems.
%
%
\appendix[Proof of Theorem~\ref{LinearGeometry.ExistenceUniquenessEquilibriumTheorem}]
%
%
Consider $\mathscr{S}\left(\nu\right)$ with $\left\|\nu\right\|_1=n$. The goal is to show that the algebraic equation
\begin{equation}\label{LinearGeometry.AlgebraicEquation}
f\left(x,u_\mathrm{ss}\right)\triangleq-\nabla_xV\left(x,u_\mathrm{ss}\right)=\mathbf{0}
\end{equation}
has exactly one solution $x_\mathrm{ss}$ in $\mathscr{S}\left(\nu\right)$ and that the Jacobian matrix of the vector field $f\left(\,\cdot\,,u_\mathrm{ss}\right)$ is negative definite over the set $\mathscr{S}\left(\nu\right)$. To that end, denote the $k$th component of $f\left(x,u_\mathrm{ss}\right)$ by $f_k\left(x,u_\mathrm{ss}\right)$ and let $u_\mathrm{ss}^j$ be the $j$th component of $u_\mathrm{ss}$, where by hypothesis $u_\mathrm{ss}^j>0$. Then using~\eqref{LinearGeometry.EnergyFunction}, $f_k\left(x,u_\mathrm{ss}\right)$ can be  written as
\begin{align}\label{LinearGeometry.ExplicitExpressionF}
f_k\left(x,u_\mathrm{ss}\right)&=-\frac{{\partial}}{{\partial}x_k}\,V\left(x,u_\mathrm{ss}\right),\nonumber\\
&=-\frac{{\partial}}{{\partial}x_k}\sum_{\substack{j=1\\j\neq{k}}}^{{n}}\frac{1}{\left|x_k-x_j\right|}
-\frac{{\partial}}{{\partial}x_k}\sum_{j=0}^{c}\frac{u_\mathrm{ss}^j}{\left|x_k-q_j\right|}\,,\nonumber\\
&=\sum_{\substack{j=1\\j\neq{k}}}^{{n}}\frac{\mathrm{sign}\left(x_k-x_j\right)}{\left(x_k-x_j\right)^2}+\sum_{j=0}^c
\frac{u_\mathrm{ss}^j\mathrm{sign}\left(x_k-q_j\right)}{\left(x_k-q_j\right)^2}\,,
\end{align}
where $\mathrm{sign}\left(\cdot\right)$ denotes the signum function. Further, for $x\in\mathscr{S}\left(\nu\right)$, the partial derivatives of $f_k\left(\,\cdot\,,u_\mathrm{ss}\right)$ with respect to $x_k$ and $x_i$, $i\neq{k}$ exist and are given by
\begin{subequations}\label{LinearGeometry.PartialDerivativeF}
\begin{align}
\frac{{\partial}f_k}{{\partial}x_k}\left(x,u_\mathrm{ss}\right)&=-\sum_{\substack{j=1\\j\neq{k}}}^{{n}}\frac{2}{\left|x_k-x_j\right|^3}
-\sum_{j=0}^{c}\frac{2u_\mathrm{ss}^j}{\left|x_k-q_j\right|^3}\,,\label{LinearGeometry.PartialDerivativeFwrtXk}\\
\frac{{\partial}f_k}{{\partial}x_i}\left(x,u_\mathrm{ss}\right)&=\frac{2}{\left|x_k-x_i\right|^3}\,,\quad{i\neq{k}}.
\label{LinearGeometry.PartialDerivativeFwrtXi}
\end{align}
\end{subequations}

Given $x\in\mathscr{S}\left(\nu\right)$, for each $k=1,2,\ldots,n$, let $k^\prime$ denote the smallest integer satisfying $x_k<q_{k^\prime}$ and define
\begin{align*}
x_1^L&=q_0,\\
x_k^L&=\max\left(x_{k-1},q_{k^\prime-1}\right),\quad{k=2,3,\ldots,n},\\
x_k^U&=\min\left(x_{k+1},q_{k^\prime}\right),\quad{k=1,2,\ldots,n-1},\\
x_n^U&=q_c.
\end{align*}
Fix all components of $x\in\mathscr{S}\left(\nu\right)$ except for $x_k$ which is allowed to vary in the segment $\left(x_k^L,x_k^U\right)$. The negative derivative~\eqref{LinearGeometry.PartialDerivativeFwrtXk} implies~that $f_k\left(\,\cdot\,,u_\mathrm{ss}\right)$ is strictly decreasing in $x_k\in\left(x_k^L,x_k^U\right)$ with all other variables fixed. In addition,~\eqref{LinearGeometry.ExplicitExpressionF} implies that this function tends to $+\infty$ and $-\infty$ as $x_k$ tends to $x_k^L$ and $x_k^U$, respectively. Thus, it is concluded that $f_k\left(x,u_\mathrm{ss}\right)=0$ has one and only one solution for $x_k$ in the interval $x_k\in\left(x_k^L,x_k^U\right)$, with all other variables fixed. This solution depends on $x_1,\ldots,x_{k-1},x_{k+1},\ldots,x_n$ and is denoted by
\begin{equation*}
g_k\left(x\right)=g_k\left(x_1,\ldots,x_{k-1},x_{k+1},\ldots,x_n\right).
\end{equation*}

Let $g:\mathbb{R}^n\to\mathbb{R}^n$ be a vector-valued function with the components $g_k$, $k=1,2,\ldots,n$. Then, $x_\mathrm{ss}$ is a solution to~\eqref{LinearGeometry.AlgebraicEquation}, if and only if it is a fixed point of the mapping $g$, i.e., if it solves
\begin{equation*}
x=g\left(x\right).
\end{equation*}
It is proven next that~\eqref{LinearGeometry.AlgebraicEquation} has exactly one solution in $\mathscr{S}\left(\nu\right)$ by showing that $g$ is a contraction map on $\mathscr{S}\left(\nu\right)$, and thereby it has exactly one fixed point in this set.

The gradient ${\nabla_x}g_k$ of the scalar function $g_k$ is obtained as follows. Based on the definition of $g_k$, the identity
\begin{equation*}
f_k\left(x,u_\mathrm{ss}\right)\big{|}_{x_k=g_k\left(x_1,\ldots,x_{k-1},x_{k+1},\ldots,x_n\right)}=0
\end{equation*}
holds true for any $\left(x_1,\ldots,x_{k-1},x_{k+1},\ldots,x_n\right)\in\mathscr{S}\left(\nu^{\,\prime}\right)$, where the components of $\nu^{\,\prime}$ are similar to $\nu$ except for $\nu_k^{\,\prime}$ which is equal to $\nu_k-1$. Differentiating this identity with respect to $x_i\neq{x_k}$ leads to
\begin{equation*}
\frac{{\partial}f_k}{{\partial}x_i}\left(x,u_\mathrm{ss}\right)+\frac{{\partial}f_k}
{{\partial}x_k}\left(x,u_\mathrm{ss}\right)\,\frac{{\partial}g_k}{{\partial}x_i}\left(x\right)=0.
\end{equation*}
Solving this equation for ${\partial}g_k/{\partial}x_i$ results in
\begin{equation}\label{LinearGeometry.PartialDerivativeG}
\frac{{\partial}g_k}{{\partial}x_i}\left(x\right)=-\left(\frac{{\partial}f_k}{{\partial}x_k}\left(x,u_\mathrm{ss}\right)\right)^{-1}
\frac{{\partial}f_k}{{\partial}x_i}\left(x,u_\mathrm{ss}\right).
\end{equation}
Since $g_k$ does not depend on $x_k$, its partial derivative with respect to $x_k$ is identically $0$ so that
\begin{equation*}
\frac{{\partial}g_k}{{\partial}x_k}\left(x\right)=0.
\end{equation*}
Substituting the explicit expressions~\eqref{LinearGeometry.PartialDerivativeF} into~\eqref{LinearGeometry.PartialDerivativeG} and noting that $u_\mathrm{ss}^j>0$, it is straightforward to verify that
\begin{equation*}
\left\|{\nabla_x}g_k\left(x\right)\right\|_1=\sum_{i=1}^n\left|\frac{{\partial}g_k}{{\partial}x_i}\left(x\right)\right|<1,{\quad}k=1,2,\ldots,n.
\end{equation*}

Let $x,y\in\mathscr{S}\left(\nu\right)$ and consider the line segment
\begin{equation*}
\ell\left(s\right)=sx+\left(1-s\right)y,{\quad}s\in\left[0,1\right].
\end{equation*}
Since $\mathscr{S}\left(\nu\right)$ is a convex set, all points on this line segment are inside the set. Applying the mean value theorem~\cite{BOOK.Apostol.67} to the scalar function $g_k\circ\ell$ implies that there exists $s_k^*\in\left(0,1\right)$ such that
\begin{align*}
g_k\circ\ell\left(1\right)-g_k\circ\ell\left(0\right)&=\frac{dg_k\circ\ell}{ds}\left(s_k^*\right)\\
&=\left({\nabla_x}g_k\left(\ell\left(s_k^*\right)\right)\right)^\mathrm{T}\ell^\prime\left(s_k^*\right).
\end{align*}
Substituting $g_k\circ\ell\left(1\right)=g_k\left(x\right)$, $g_k\circ\ell\left(0\right)=g_k\left(y\right)$, and $\ell^\prime\left(s_k^*\right)=x-y$ into this equality and taking absolute values of its sides lead to
\begin{align*}
\left|g_k\left(x\right)-g_k\left(y\right)\right|&=\left|\left({\nabla_x}
g_k\left(\ell\left(s_k^*\right)\right)\right)^\mathrm{T}\left(x-y\right)\right|\\
&\leqslant\left\|{\nabla_x}g_k\left(\ell\left(s_k^*\right)\right)\right\|_1\left\|x-y\right\|_\infty.
\end{align*}
Because this inequality holds for every $k=1,2,\ldots,n$, it is concluded that
\begin{equation*}
\left\|g\left(x\right)-g\left(y\right)\right\|_\infty\leqslant{K}\left\|x-y\right\|_\infty,
\end{equation*}
where
\begin{equation*}
K=\max_{k=1,2,\ldots,n}\left\|{\nabla_x}g_k\left(\ell\left(s_k^*\right)\right)\right\|_1<1.
\end{equation*}
This verifies that $g$ is a contraction map.

Finally, it is shown that $V\left(\,\cdot\,,u_\mathrm{ss}\right)$ is strictly convex on $\mathscr{S}\left(\nu\right)$, and thereby the solution $x_\mathrm{ss}$ to~\eqref{LinearGeometry.AlgebraicEquation} is a stable equilibrium. The Ger{\v{s}}gorin circle theorem~\cite{BOOK.Varga.04} applied to the Jacobian matrix~of the vector field $f\left(\,\cdot\,,u_\mathrm{ss}\right)$ implies that the eigenvalues $\varrho $ of this matrix are inside the circles of the form
\begin{equation*}
\left|\,\varrho-\frac{{\partial}f_k}{{\partial}x_k}\left(x,u_\mathrm{ss}\right)\right|\leqslant\sum_{\substack{i=1\\i\neq{k}}}^{{n}}
\left|\frac{{\partial}f_k}{{\partial}x_i}\left(x,u_\mathrm{ss}\right)\right|.
\end{equation*}
Since ${\partial}f_k/{\partial}x_k<0$ according to~\eqref{LinearGeometry.PartialDerivativeF} and the right-hand side of the inequality is smaller than $\left|{\partial}f_k/{\partial}x_k\right|$, the eigenvalues of the Jacobian matrix have negative values.
%
%

%
%
\end{document}